\documentclass[11pt]{article}
\usepackage{amssymb}
\usepackage{amsmath}
\usepackage{dmvnbase}
\usepackage[dvips]{graphics}
\usepackage[left=3cm, right=3cm, top=2cm, bottom=2cm]{geometry}
\newtheorem{theorem}{Theorem}[section]
\newtheorem{lem}[theorem]{Lemma}
\newtheorem{df}[theorem]{Definition}
\newtheorem{prop}[theorem]{Proposition}
\newtheorem{cor}[theorem]{Corollary}

\newtheorem{conj}[theorem]{Conjecture}

\numberwithin{equation}{section}
\renewcommand{\theequation}{{\thesection}.{\arabic{equation}}}

\def\beq{\begin{equation}}
\def\eeq{\end{equation}}
\def\pbeq{\begin{equation*}}
\def\peeq{\end{equation*}}
\def\be{\beq\begin{array}{c}}
\def\ee{\end{array}\eeq}

\def\A{\rm A}
\def\Ab{\bar{\rm A}}

\def\bt{\,\hbox{\small$\boxtimes$}\,}

\def\comp{\ts{\scriptstyle\circ}\ts}

\def\d{\partial}

\def\Eh{\hat E}
\def\Ep{E^{\ts\prime}}

\def\gl{\glg}

\def\G{\Gc\hr{\Cbb^m \otimes \Cbb^n}}
\def\GD{\Gc\Dc\hr{\Cbb^m \otimes \Cbb^n}}
\def\GL{\operatorname{GL}}

\def\h{\mathfrak{h}}
\def\H{\Hc\hr{\Cbb^m \otimes \Cbb^n}}
\def\HD{\Hc\Dc\hr{\Cbb^m \otimes \Cbb^n}}

\def\I{{\rm I}}
\def\Ib{\bar{\rm I}}

\def\J{{\rm J}}
\def\Jb{\bar{\rm J}}
\def\Jp{{\rm J}^\prime}
\def\Jpb{\bar{\rm J}^\prime}

\def\ka{\kappa}

\def\la{\lambda}

\def\lcd{\ts,\ldots,}
\def\le{\leqslant}

\def\n{\mathfrak{n}}
\def\np{\n^{\ts\prime}}

\def\ov{\overline}

\def\p{\mathfrak{p}}
\def\phi{\varphi}
\def\P{\Pc\hr{\Cbb^m \otimes \Cbb^n}}
\def\PD{\Pc\Dc\hr{\Cbb^m \otimes \Cbb^n}}

\def\q{\mathfrak{q}}
\def\qp{\q^{\ts\prime}}

\def\s{\sigma}
\def\sl{\slg}
\def\so{\mathfrak{so}}
\def\sp{\mathfrak{sp}}

\def\th{\theta}
\def\ts{\hskip1pt}
\def\Tp{T^{\ts\prime}}

\def\U{\operatorname{U}}
\def\Uhb{\,\overline{\!\U(\h)\!\!\!}\,\,\,}

\def\TPhi{\widetilde\Phi}

\def\xic{\check\xi}

\def\Y{\operatorname{Y}}

\def\Zbb{\mathbb{Z}}

\newenvironment{proof}{{\flushleft\it Proof:}} {\hfill$\square$ \\}

\begin{document}


\begin{center}
{\Large\bf Rational representations of the Yangian $\Y(\gl_n)$}
\bigskip
\bigskip

\noindent
A. Shapiro
\bigskip
\medskip

{
\footnotesize
\textit{Department of Mathematics, University of California, Berkeley, CA, 94720;} \\
\smallskip
\textit{Institute of Theoretical \& Experimental Physics, 117259, Moscow, Russia;} \\
\medskip
alexander.m.shapiro@gmail.com
}

\end{center}
\bigskip

\begin{abstract}
We construct a series of rational representations of $\Y(\gl_n)$ and intertwining operators between them.
We find explicit expressions for the images of highest-weight vectors under the intertwining operators.
Finally, we state a conjecture that all irreducible finite-dimensional rational $\Y(\gl_n)$-modules arise
as images of the constructed intertwining operators.
\end{abstract}
\bigskip


\section*{\bf\normalsize 0.\ Introduction}

There are two classical approaches to the representation theory of the Lie algebra $\gl_n$. The first one, known as the theory of highest weight, suggests a parametrization of all complex finite-dimensional irreducible $\gl_n$-modules by $n$-tuples of complex numbers $\la = \hr{\la_1, \dots, \la_n}$. These $n$-tuples are called weights and describe the action of the Cartan subalgebra $\h \in \gl_n$ on the module. The second approach is based on the Schur-Weyl duality between the representations of the symmetric group $\Sg_m$ and the Lie algebra $\gl_n$. It produces a set of $\gl_n$-modules called the Schur modules, which are certain submodules (or quotient modules) of tensor products of the defining $\gl_n$-module $\Cbb^n$. Let us borrow terminology from the representation theory of the Lie group $\GL_n$, and call a $\gl_n$-module polynomial (respectively rational) if it is isomorphic to a submodule of $\Cbb^{\otimes N}$ (respectively $\Cbb^{\otimes N_1} \otimes (\Cbb^*)^{\otimes N_2}$). Then, the Schur modules exhaust all the polynomial $\gl_n$-modules.

Similar approaches can be taken towards the representation theory of the Yangian $\Y(\gl_n)$. The highest weight theory for Yangians was developed by Drinfeld. Let us call two finite-dimensional representations of the algebra $\Y(\gl_n)$ similar if they differ by an automorphism of the form \eqref{fut}. Up to similarity the irreducible finite-dimensional $Y(\gl_n)$-modules were classified in \cite{D3}. Due to this classification every irreducible finite-dimensional $\Y(\gl_n)$-module is described by a set of Drinfeld polynomials, serving the same goals as weights in the representation theory of $\gl_n$. Later on, analogous results on the representations of shifted Yangians and $W$-algebras were obtained in \cite{BK}.

Another approach to the representation theory of $\Y(\gl_n)$ was considered in the works \cite{KN1,KN2} and is based on the $(\GL_m,\,\gl_n)$ Howe duality (see \cite{H1,H2}). The main role is played by a certain functor $\Ec_m$ from the category of $\gl_m$-modules to the category of $\Y(\gl_n)$-modules. This functor appeared as a composition of Drinfeld (see \cite{D2}) and Cherednik (see \cite{C,AST}) functors, and can be regarded as a reformulation of the Olshanski centralizer construction (see \cite{O1,O2}). Then the following steps were taken in \cite{KN1,KN2}. First, the functor $\Ec_m$ was applied to the Verma modules of the algebra $\gl_m$ which gave rise to a series of \emph{standard} representations of the Yangian $\Y(\gl_n)$. Second, with the help of the theory of Zhelobenko operators for Mickelsson algebras (see \cite{Z1,Z2,K,KO}) certain intertwining operators between the standard $\Y(\gl_n)$-modules were constructed. Finally, in the works \cite{KN5,KNP} it was shown that all irreducible finite-dimensional $\Y(\gl_n)$-modules considered up to similarity can be obtained as the images of the intertwining operators constructed in \cite{KN1,KN2}. An analogous result for representations of quantum affine algebras appeared earlier in \cite{AK}.

In analogy with the $\gl_n$-case, let us call a representation of the Yangian $\Y(\gl_n)$ \emph{polynomial} if it is a subquotient of a tensor product of vector representations of $\Y(\gl_n)$. Note that all $\Y(\gl_n)$-modules constructed in \cite{KN1,KN2} are polynomial. Moreover, any polynomial $\Y(\gl_n)$-module appears this way. Let us call a representation of the Yangian $\Y(\gl_n)$ \emph{rational} if it is a subquotient of a tensor product of vector and dual vector representations of $\Y(\gl_n)$. Earlier, the irreducible rational representations of $\Y(\gl_n)$ associated with skew Young diagrams were investigated by Nazarov in \cite{N}. In the present paper we try to generalize the approach of Khoroshkin and Nazarov in order to obtain rational representations of the algebra $\Y(\gl_n)$.

First, we consider a modification $\Ec_{p,q}$ of the functor $\Ec_m$, based on the $(\U_{p,q},\,\gl_n)$ Howe duality (see \cite{EHW}, \cite{EP}, \cite{KV}). Second, we apply the functor $\Ec_{p,q}$ to the Verma modules of $\gl_m$ which gives certain \emph{standard rational} modules. Third, with the help of technique developed for twisted Yangians $\Y(\so_{2n})$, $\Y(\sp_{2n})$ in \cite{KN3}, \cite{KN4} we construct intertwining operators between obtained tensor products and compute the images of highest-weight vectors under intertwining operators. Next, using the results of \cite{KN5}, we observe that (under some conditions on the parameters of standard rational modules) the image of the certain intertwining operator is an irreducible rational $\Y(\gl_n)$-module. Finally, we state as a conjecture that all irreducible rational $\Y(\gl_n)$-modules can be obtained by this construction. We return to this problem in the forthcoming publication \cite{KNS}.

\section*{\bf\normalsize 1.\ Basics}
\setcounter{section}{1}
\setcounter{theorem}{0}

\subsection*{\it\normalsize 1.1.\ Yangian $\Y(\gl_n)$}

The Yangian $\Y(\gl_n)$ is a deformation in the class of Hopf algebras of the universal enveloping algebra of the Lie algebra $\gl_n[u]$ of polynomial current, see for instance \cite{D1}. The unital associative algebra $\Y(\gl_n)$ is generated by the family
$$
T_{ij}^{(1)},T_{ij}^{(2)}, \ldots \quad\text{where}\quad i,j=1,\dots,n.
$$
Consider the generating functions
\begin{equation}
\label{tser}
T_{ij}(u)=\de_{ij}+T_{ij}^{(1)}u^{-1}+T_{ij}^{(2)}u^{-2}+\dots\in\Y(\gl_n)[[u^{-1}]]
\end{equation}
with formal parameter $u$. Then the defining relations in $\Y(\gl_n)$ can be written as
\begin{equation}
\label{yangrel}
(u-v)\cdot[T_{ij}(u),T_{kl}(v)]=T_{kj}(u)T_{il}(v)-T_{kj}(v)T_{il}(u)
\end{equation}
where $i, j, k, l=1,\dots,n$.

The above relations imply that for any $z\in\Cbb$ assignments
\begin{equation}
\label{tauz}
\tau_z\colon T_{ij}(u)\mapsto\,T_{ij}(u-z) \quad\textrm{for}\quad i,j=1,\dots, n
\end{equation}
define an automorphism $\tau_z$ of the algebra $\Y(\gl_n)$. Here each of the formal series $T_{ij}(u-z)$ in $(u-z)^{-1}$ should be re-expanded in $u^{-1}$, then the assignment \eqref{tauz} is a correspondence between the respective coefficients of series in $u^{-1}$.

Now let $E_{ij}\in\gl_n$ with $i,j=1,\dots, n$ be the standard matrix units. Sometimes $E_{ij}$ will also denote elements of the algebra $\End(\Cbb^n)$ but this should not cause any confusion. The Yangian $\Y(\gl_n)$ contains the universal enveloping algebra $\U(\gl_n)$ as a subalgebra, the embedding $\U(\gl_n)\to\Y(\gl_n)$ can be defined by the assignments
$$
E_{ij}\mapsto T_{ij}^{(1)} \quad\text{for}\quad i,j=1,\dots,n.
$$
Moreover, there is a homomorphism $\pi_n \colon \Y(\gl_n)\to\U(\gl_n)$ which is identical on the subalgebra $\U(\gl_n)\subset\Y(\gl_n)$ and is given by
\begin{equation}
\label{pin}
\pi_n\colon T_{ij}^{(2)},T_{ij}^{(3)},\dots\,\mapsto\,0     \quad\text{for}\quad    i,j=1,\dots,n.
\end{equation}

Let $T(u)$ be an $n\times n$ matrix whose $i, j$ entry is the series $T_{ij}(u)$. The relations \eqref{yangrel} can be rewritten by means of the \emph{Yang $R$-matrix}
\begin{equation}
\label{ru}
R(u)\,=\,1\otimes1-\sum_{i,j=1}^n\frac{E_{ij}\otimes E_{ji}}{u}
\end{equation}
where the tensor factors $E_{ij}$ and $E_{ji}$ are regarded as $n\times n$ matrices. Note that
\begin{equation}
\label{rur}
R(u)\,R(-u)=1-\frac{1}{u^2}.
\end{equation}
Consider two $n^2\times n^2$ matrices whose entries are series with coefficients in the algebra $\Y(\gl_n)$,
$$
T_1(u)=T(u)\otimes1
\ \quad\text{and}\ \quad
T_2(v)=1\otimes T(v)\,.
$$
Then a collection of relations \eqref{yangrel} for all possible indices $i, j, k, l$ is equivalent to
\begin{equation}
\label{rtt}
R(u-v)\,T_1(u)\,T_2(v)\,=\,T_2(v)\,T_1(u)\,R(u-v)\,.
\end{equation}

Further, the Yangian $\Y(\gl_n)$ is a Hopf algebra over the field $\Cbb$. We define the comultiplication $\De:\Y(\gl_n)\to\Y(\gl_n)\otimes\Y(\gl_n)$ by the assignment
\beq \label{1.33}
\De:\,T_{ij}(u)\mapsto\sum_{k=1}^n\ T_{ik}(u)\otimes T_{kj}(u)\,.
\eeq
Throughout the article this comultiplication will be used for tensor products of $\Y(\gl_n)$-modules. The counit homomorphism
$\ep:\Y(\gl_n)\to\Cbb$ is defined by
$$
\ep\colon T_{ij}(u)\mapsto\de_{ij}\cdot1.
$$
The antipode ${\rm S}$ on $\Y(\gl_n)$ is given by
$$
{\rm S}\colon T(u)\mapsto T(u)^{-1}
$$
and defines an anti-automorphism of the associative algebra $\Y(\gl_n)$.

Let $\Tp(u)$ be the transpose to the matrix $T(u)$. Then the $i, j$ entry of the matrix $\Tp(u)$ is $T_{ji}(u)$. Consider $n^2\times n^2$ matrices
$$
\Tp_1(u)=\Tp(u)\otimes1 \ \quad\text{and}\ \quad \Tp_2(v)=1\otimes\Tp(v)\,.
$$
Note that the Yang $R$-matrix \eqref{ru} is invariant under applying the transposition to both tensor factors. Hence the relation \eqref{rtt} implies
$$
\Tp_1(u)\,\Tp_2(v)\,R(u-v)
\,=\,
R(u-v)\,\Tp_2(v)\,\Tp_1(u)\,,
$$
\begin{equation}
\label{rttp}
R(u-v)\,\Tp_1(-u)\,\Tp_2(-v)
\,=\,
\Tp_2(-v)\,\Tp_1(-u)\,R(u-v)\,.
\vspace{4pt}
\end{equation}
To obtain the latter relation we used \eqref{rur}. By comparing the relations \eqref{rtt} and \eqref{rttp}, an involutive automorphism of the algebra $\Y(\gl_n)$ can be defined by the assignment
$$
\om\colon T(u)\mapsto\Tp(-u),
$$
understood as a correspondence between the respective matrix entries. For further details on the algebra $\Y(\gl_n)$ see \cite[Chapter 1]{MNO}.

\subsection*{\it\normalsize 1.2.\ Representations of Yangian $\Y(\gl_n)$}

Let $\Phi$ be an irreducible finite-dimensional $\Y(\gl_n)$-module. A non-zero vector $\phi \in \Phi$ is said to be of \emph{highest weight} if it is annihilated by all the coefficients of the series $T_{ij}(u)$ with $1 \le i<j \le n$ and is an eigenvector for all the coefficients of the series $T_{ii}(u)$ for $1 \le i \le n$. In that case $\phi$ is unique up to a scalar multiplier and for $i=1, \dots, n-1$ holds
$$
T_{ii}(u)T_{i+1,i+1}(u)^{-1}\phi = P_i(u+\frac12)P_i(u-\frac12)^{-1}\phi
$$
where $P_i(u)$ is a monic polynomial in $u$ with coefficients in $\Cbb$. Then $P_1(u), \dots, P_{n-1}(u)$ are called the \emph{Drinfeld polynomials\/} of $\Phi$. Any sequence of $n-1$ monic polynomials with complex coefficients arises this way. An irreducible finite-dimensional $\Y(\gl_n)$-module is defined by the set of eigenfunctions $\La_{ii}(u)$ such that
$$
T_{ii}(u)\ph = \La_{ii}(u)\ph
$$
for $i=1 \lcd n$. Thus, an irreducible finite-dimensional $\Y(\gl_n)$-module is defined by a set of polynomials $P_1(u), \dots, P_{n-1}(u)$ and some normalizing factor, for example $\La_{nn}(u)$.

Relations \eqref{yangrel} show that for any formal power series $g(u)$ in $u^{-1}$ with coefficients in $\Cbb$ and leading term $1$, the assignments
\beq
\label{fut}
T_{ij}(u)\mapsto g(u)T_{ij}(u)
\eeq
define an automorphism of the algebra $\Y(\gl_n)$. The subalgebra in $\Y(\gl_n)$ consisting of all elements which are invariant under every automorphism of the form \eqref{fut}, is called the \emph{special Yangian\/} of $\gl_n$. The special Yangian of $\gl_n$ is a Hopf subalgebra of $\Y(\gl_n)$ and is isomorphic to the Yangian $\Y(\sl_n)$ of the special linear Lie algebra $\sl_n\subset\gl_n$ considered in \cite{D2,D3}. For the proofs of the latter two assertions see \cite[Subsection 1.8]{M}.

Two irreducible finite-dimensional $\Y(\gl_n)$-modules are called \emph{similar} if their restrictions to the special Yangian are isomorphic. Therefore, irreducible finite-dimensional $\Y(\sl_n)$-modules are defined by a set of Drinfeld polynomials, while irreducible finite-dimensional $\Y(\gl_n)$-modules are parameterized by their Drinfeld polynomials only up to similarity. For further details on representations of $\Y(\gl_n)$ see \cite{D3,CP}.

Now, let us define the fundamental representation $V_z$ and the dual fundamental representation $V'_z$ of $\Y(\gl_n)$. As a vector space $V_z = V'_z = \Cbb^n$, the corresponding actions are given by the pull-backs of the defining $\gl_n$-modules through the evaluation and the dual evaluation homomorphisms respectively
\begin{align*}
\pi_z  = \pi_n\comp\tau_{-z}\colon     &\Y(\gl_n)\longra\U(\gl_n), &\hspace{-40pt}T_{ij}(u) &\mapsto \de_{ij} + \frac{E_{ij}}{u+z}, \\
\pi'_z = \pi_n\comp\tau_z\comp\om\colon &\Y(\gl_n)\longra\U(\gl_n), &\hspace{-40pt}T_{ij}(u) &\mapsto \de_{ij} - \frac{E_{ji}}{u+z}.
\end{align*}

Let also $\Om_z$ and $\Om'_z$ denote one-dimensional representations of $\Y(\gl_n)$ defined as the pull-backs of the standard action of $\gl_n$ on $\La^n(\Cbb^n)$ through the evaluation and the dual evaluation homomorphisms. Thus,
$$
T_{ij}(u) \mapsto \de_{ij}\cdot\frac{u+z+1}{u+z}
\qquad\text{and}\qquad
T_{ij}(u) \mapsto \de_{ij}\cdot\frac{u+z-1}{u+z}
$$
on $\Om_z$ and $\Om'_z$ correspondingly.

\begin{df}
\begin{itemize}
\item[a)] Representation of Yangian $\Y(\gl_n)$ is called polynomial if it is isomorphic to a subquotient of tensor product of fundamental representations $V_z$ with arbitrary values of $z$.
\item[b)] Representation of Yangian $\Y(\gl_n)$ is called rational if it is isomorphic to a subquotient of tensor product of fundamental and dual fundamental representations $V_z$ and $V'_z$ with arbitrary values of $z$.
\end{itemize}
\end{df}

Note that representations $V_z$ and $\Om_z$ are polynomial while representations $V'_z$ and $\Om'_z$ are rational. We would also like to point out that the modules $\Om_z$ and $\Om'_z$ are cocentral, i.e. in tensor products of $\Y(\gl_n)$-modules one can permute them with other modules. More precisely, the form of the comultiplication map $\De$ implies that for any $\Y(\gl_n)$-module $M$ there is a pair of canonical isomorphism
$$
M \otimes \Om_z \cong \Om_z \otimes M \qquad\text{and}\qquad M \otimes \Om'_z \cong \Om'_z \otimes M
$$
sending $m \otimes a \mapsto a \otimes m$ where $m \in M$ and $a \in \Om_z$ or $a \in \Om'_z$.

\subsection*{\it\normalsize 1.3.\ Functor}

Let $E$ be an $m \times m$ matrix whose $a,b$ entry is the generator $E_{ab}\in\gl_m$, and let $\Ep$ be its transpose. Consider a matrix
\beq \label{Xu}
X(u) = \hr{u+\th\Ep}^{-1} \qquad\text{with}\qquad \th=\pm1
\eeq
whose $a,b$ entry is a formal power series in $u^{-1}$
\beq \label{Xabu}
X_{ab}(u) = u^{-1}\hr{\de_{ab} + \suml{s=0}{\infty} X_{ab}^{(s)} u^{-s-1}}.
\eeq
For $a,b=1,\dots,m$ elements $X_{ab}^{(s)}\in\U(\gl_m),$ moreover
$$
X_{ab}^{(0)} = \; -\th E_{ba} \qquad\text{and}\qquad
X_{ab}^{(s)} = \suml{c_1,\dots,c_s=1}{m} (-\th)^{s+1} E_{c_1 a} E_{c_2 c_1} \dots E_{c_s c_{s-1}} E_{b c_s} \quad\text{for}\quad s\ge1.
$$

Consider the ring $\P$ of polynomial functions on $\Cbb^m\otimes\Cbb^n$ with coordinate functions $x_{ai}\,$ where $\,a=1,\dots,m\,$ and $\,i=1,\dots,n$. Let $\PD$ be the ring of differential operators with polynomial coefficients on $\P$, and let $\d_{ai}$ be the partial derivation corresponding to $x_{ai}$.

Consider also the Grassmann algebra $\G$ of the vector space $\Cbb^m\otimes\Cbb^n$. It is generated by the elements $x_{ai}$ subject to the anticommutation relations $x_{ai}x_{bj} = -x_{bj}x_{ai}$ for all indices $a,b=1,\dots,m$ and $i,j=1,\dots,n$. Let $\d_{ai}$ be the operator of left derivation on $\G$ corresponding to the variable $x_{ai}$. Let $\GD$ denote the ring of $\Cbb$-endomorphisms of $\G$ generated by all operators of left multiplication $x_{ai}$ and by all operators $\d_{ai}$.

Let us define
\begin{align*}
\H&=\P & \quad&\text{and}\quad & \HD&=\PD & \quad&\text{if}\quad & \th&=1, \\
\H&=\G & \quad&\text{and}\quad & \HD&=\GD & \quad&\text{if}\quad & \th&=-1.
\end{align*}
Therefore, algebra $\H$ is generated by the elements $x_{ai}$, $\;a=1,\dots,m, \; i=1,\dots,n$ subject to relations
$$
x_{ai}x_{bj} -\th x_{bj}x_{ai} = 0.
$$
Algebra $\HD$ is generated by the elements $x_{ai}$ and $\d_{bj}$, $\;a,b=1,\dots,m, \; i,j=1,\dots,n$ subject to relations
\beq \label{rel_dx}
\begin{split}
x_{ai}x_{bj}   \,-\, \th x_{bj}x_{ai}   \;&=\; 0, \\
\d_{ai}\d_{bj} \,-\, \th \d_{bj}\d_{ai} \;&=\; 0, \\
\d_{ai}x_{bj}  \,-\, \th x_{bj}\d_{ai}  \;&=\; \de_{ab}\de_{ij}.
\end{split}
\eeq
Note that $\th=1$ corresponds to the case of commuting variables, while $\th=-1$ corresponds to the case of anticommuting variables.

From now on and till the end of the paper we assume $m=p+q$, where $p$ and $q$ are non-negative integers. Let us introduce new coordinates
\beq \label{new_coord}
p_{ci} =
\left\{
\begin{array}{rl}
-\th x_{ci}, &\text{for}\;\, c=1,\dots,p\\
\d_{ci}, &\text{for}\;\, c=p+1,\dots,m
\end{array}
\right.
\qquad\quad
q_{ci} =
\left\{
\begin{array}{rl}
\d_{ci}, &\text{for}\;\, c=1,\dots,p\\
x_{ci},  &\text{for}\;\, c=p+1,\dots,m.
\end{array}
\right.
\eeq
Now relations \eqref{rel_dx} can be rewritten in the following form
\begin{align*}
q_{ai}q_{bj} \,-\, \th q_{bj}q_{ai} \;&=\; 0, \\
p_{ai}p_{bj} \,-\, \th p_{bj}p_{ai} \;&=\; 0, \\
p_{ai}q_{bj} \,-\, \th q_{bj}p_{ai} \;&=\; \de_{ab}\de_{ij}.
\end{align*}
Define the elements $\Eh_{ai,bj}\in\HD$ as
\beq \label{Eaibj}
\Eh_{ai,bj} = q_{ai}p_{bj}.
\eeq
Elements $\Eh_{ai,bj}$ satisfy relations
\begin{align}
\label{E_rel1}
\hs{\Eh_{ai,bj}, \Eh_{ck,dl}} &= \de_{bc}\de_{jk}\Eh_{ai,dl} - \de_{ad}\de_{il}\Eh_{ck,bj}, \\
\label{E_rel2}
\Eh_{ai,bj}\Eh_{ck,dl} - \th \Eh_{ck,bj}\Eh_{ai,dl} &= \de_{bc}\de_{jk}\Eh_{ai,dl} - \th\de_{ab}\de_{ij}\Eh_{ck,dl}
\end{align}
which also imply
\beq
\label{E_rel3}
\Eh_{ck,dl}\Eh_{ai,bj} - \th \Eh_{ck,bj}\Eh_{ai,dl} = \de_{ad}\de_{il}\Eh_{ck,bj} - \th\de_{ab}\de_{ij}\Eh_{ck,dl}.
\eeq

There is an action of the algebra $\gl_m$ on the space $\H$ which is defined by homomorphism $\zeta_n \colon \U(\gl_m) \mapsto \HD$, where
\beq\label{zeta}
\zeta_n\hr{E_{ab}} = \th\de_{ab}\dfrac n2 + \suml{k=1}{n} \Eh_{ak,bk}.
\eeq
The homomorphism property can be verified using the relation \eqref{E_rel1}. Hence, there exists an embedding $\U(\gl_m)\hookrightarrow\U(\gl_m)\otimes\HD$ defined for $a,b = 1,\dots,m$ by the mappings
\beq
\label{gl_emb}
E_{ab} \mapsto E_{ab}\otimes1 + 1\otimes \zeta_n(E_{ab}).
\eeq

\begin{prop} \label{prop_bimod}
\hskip5pt$\rm i)$ One can define a homomorphism $\al_m \colon \Y(\gl_n) \to \U(\gl_m)\otimes\HD$ by mapping
\beq
\label{yang_hom}
\al_m \colon\hskip5pt T_{ij}(u) \,\mapsto\, \de_{ij} + \suml{a,b=1}{m} X_{ab}(u) \otimes \Eh_{ai,bj}.
\eeq
$\rm ii)$ The image of $\Y(\gl_n)$ under the homomorphism \eqref{yang_hom} commutes with the image of $\U(\gl_m)$ under the embedding \eqref{gl_emb}.
\end{prop}

Consider an automorphism of the algebra $\HD$ such that for all $a=1,\dots,m$ and $i=1,\dots,n$
\beq \label{autom}
x_{ai} \mapsto q_{ai} \qquad\text{and}\qquad \d_{ai} \mapsto p_{ai}.
\eeq
A proof of Proposition~\ref{prop_bimod} can be obtained by applying the automorphism~\eqref{autom} to the results of \cite[Proposition 1.3]{KN1} if $\th=1$ and to the results of \cite[Proposition 1.3]{KN2} if $\th=-1$. In the appendix we give an explicit proof of the Proposition~\ref{prop_bimod}.

Finally, let $V$ be an arbitrary $\gl_m$-module. Let us define a $\U(\gl_m)\otimes\HD$-module
\beq \label{Epq}
\Ec_{p,q}(V) = V \otimes \H.
\eeq
The results of proposition \ref{prop_bimod} turns $\Ec_{p,q}$ into a functor from the category of $\gl_m$-modules to the category of $\gl_m$ and $\Y(\gl_n)$ bimodules, where the actions of the algebras $\gl_m$ and $\Y(\gl_n)$ are defined by homomorphisms \eqref{gl_emb} and \eqref{yang_hom} respectively.


\section*{\bf\normalsize 2.\ Reduction to standard modules}
\setcounter{section}{2}
\setcounter{theorem}{0}
\setcounter{equation}{0}

\subsection*{\it\normalsize 2.1.\ Parabolic induction}

Let $l,\,l_1,l_2$ be three positive integers such that $l=l_1+l_2$. Consider a $\gl_l$-module $U$, then $\Ec_{l_1,l_2}(U)$ is a $\Y(\gl_n)$-module. For any $z\in\Cbb$ denote by $\Ec_{l_1,l_2}^z(U)$ the $\Y(\gl_n)$-module obtained from  $\Ec_{l_1,l_2}(U)$ via the pull-back through the automorphism $\tau_{-\th z}$ of $\Y(\gl_n)$. In other words, the underlying vector space of $\Ec_{l_1,l_2}^z(U)$ is the same as of $\Ec_{l_1,l_2}(U)$, but the action of $T_{ij}(u)$ on $\Ec_{l_1,l_2}^z(U)$ is given by the same formula as the action of $T_{ij}(u+\th z)$ on $\Ec_{l_1,l_2}(U)$. Note that as a $\gl_l$-module $\Ec_{l_1,l_2}^z(U)$ coincides with $\Ec_{l_1,l_2}(U)$.

Decomposition $\Cbb^{m+l}=\Cbb^m\oplus\Cbb^l$ determines an embedding of the direct sum $\gl_m\oplus\gl_l$ into $\gl_{m+l}$. As a subalgebra of $\gl_{m+l}$, the direct summand $\gl_m$ is spanned by the matrix units $E_{ab}\in\gl_{m+l}$ where $a,b=1,\dots,m$, the direct summand $\gl_l$ is spanned by the matrix units $E_{ab}$ where $a,b=m+1,\dots,m+l$. Let $\q$ and $\qp$ be the Abelian subalgebras of $\gl_{m+l}$ spanned respectively by matrix units $E_{ba}$ and $E_{ab}$ for all $a=1, \dots, m$ and $b=m+1, \dots, m+l$. Define a maximal parabolic subalgebra $\p=\gl_m\oplus\gl_l\oplus\qp$ of the reductive Lie algebra $\gl_{m+l}$, then $\gl_{m+l}=\q\oplus\p$.

Consider a $\gl_m$-module $V$ and a $\gl_l$-module $U$. Let us turn the $\gl_m \oplus \gl_l$-module $V \otimes U$ into a $\p$-module by letting the subalgebra $\qp\subset\p$ act by zero. Define $V \bt U$ to be the $\gl_{m+l}$-module induced from the $\p$-module $V\otimes U$. The $\gl_{m+l}$-module $V\bt U$ is called \textit{parabolically induced} from the $\gl_m\oplus\gl_l$-module $V\otimes U$.

Now consider bimodules $\Ec_{p,q+r}(V\bt U)$ and $\Ec_{p+r,q}(V\bt U)$ over $\gl_{m+r}$ and $\Y(\gl_n)$, which are parabolically induced from the $\gl_p\oplus\gl_{q+r}$ and $\gl_{p+r}\oplus\gl_q$-module $V\otimes U$ respectively. The action of $\Y(\gl_n)$ commutes with the action of the Lie algebra $\gl_{m+r}$, and hence with the action of the subalgebra $\q\subset\gl_{m+r}$. For any $\gl_m$-module $W$ denote by $W_{\q}$ the vector space $W/\,\q\cdot W$ of the coinvariants of the action of the subalgebra $\q\subset\gl_m$ on $W$. Then the vector spaces $\Ec_{p,q+r}(V\bt U)_{\q}$ and $\Ec_{p+r,q}(V\bt U)_{\q}$ are quotients of the $\Y(\gl_n)$-modules $\Ec_{p,q+r}(V\bt U)$ and $\Ec_{p+r,q}(V\bt U)$ respectively. Note that the subalgebras $\gl_p\oplus\gl_{q+r}$ and $\gl_{p+r}\oplus\gl_q$ also act on these quotient spaces.

\begin{theorem}
\label{parind}
\hskip5pt$\rm i)$
The bimodule $\Ec_{p,q+r}(V\bt U)_\q$ over the Yangian $\Y(\gl_n)$ and the direct sum $\gl_p\oplus\gl_{q+r}$ is isomorphic to the tensor product $\Ec_{p,q}(V)\otimes\Ec_{0,r}^{m}(U)$.
\vspace{5pt} \\
$\rm ii)$
The bimodule $\Ec_{p+r,q}(V\bt U)_\q$ over the Yangian $\Y(\gl_n)$ and the direct sum $\gl_{p+r}\oplus\gl_{q}$ is isomorphic to the tensor product $\Ec_{r,0}(V)\otimes\Ec_{p,q}^{r}(U)$.
\end{theorem}

Theorem \ref{parind} is equivalent to \cite[Theorem 2.1]{KN1} under the action of automorphism~\eqref{autom} if $\th=1$ and to \cite[Theorem 2.1]{KN2} under the action of automorphism~\eqref{autom} if $\th=-1$. In both cases Theorem~2.1 was proved by establishing a linear map
$$
\chi \colon \Ec_m(V) \otimes \Ec_l^m(U) \to \Ec_{m+l}(V\bt U)_\q.
$$
Both the source and the target are bimodules over the algebras $\gl_m \oplus \gl_l$ and $\Y(\gl_n)$, while $\chi$ is a bijective map intertwining actions of algebras. One can easily show that the map $\chi$ commutes with the automorphism~\eqref{autom}, hence the intertwining property follows in our case. A proof that the map $\chi$ is bijective can be almost word by word taken from \cite{KN1} or \cite{KN2} though one has to keep in mind, that the automorphism~\eqref{autom} alters the filtration of the algebra $\H$ described in the papers just mentioned. Thus, one should consider descending filtrations
\begin{align*}
\mathop{\oplus}\limits_{N=K}^{\infty} \Pc(\Cbb^m)&\otimes\Pc^N(\Cbb^r)&
&\text{and}&
\mathop{\oplus}\limits_{N=K}^{\infty} \Pc^N(\Cbb^m)&\otimes\Pc(\Cbb^r)&
&\text{if} \quad \th=1, \\
\mathop{\oplus}\limits_{N=K}^{r} \Gc(\Cbb^m)&\otimes\Gc^N(\Cbb^r)&
&\text{and}&
\mathop{\oplus}\limits_{N=K}^{m} \Gc^N(\Cbb^m)&\otimes\Gc(\Cbb^r)&
&\text{if} \quad \th=-1
\end{align*}
for cases i) and ii) of the Theorem~\ref{parind} respectively. Therefore, the proof of the theorem follows in our case.

Let us consider the triangular decomposition
\beq
\label{tridec}
\gl_m=\n\oplus\h\oplus\np.
\eeq
of the Lie algebra $\gl_m$. Here $\h$ is the Cartan subalgebra of $\gl_m$ with the basis vectors $E_{11},\dots,E_{mm}$. Further, $\n$ and $\np$ are the nilpotent subalgebras spanned respectively by the elements $E_{ba}$ and $E_{ab}$ for all $a,b=1,\dots,m$ such that $a<b$. Denote by $V_{\n}$ the vector space $V/\,\n\cdot V$ of the coinvariants of the action of the subalgebra $\n\subset\gl_m$ on $V$. Note that the Cartan subalgebra $\h\subset\gl_m$ acts on the vector space $V_{\n}$. Now consider the bimodule $\Ec_{p,q}(V)$. The action of $\Y(\gl_n)$ on this bimodule commutes with the action of the Lie algebra $\gl_m$, and hence with the action of the subalgebra $\n\subset\gl_m$. Therefore, the space $\Ec_{p,q}(V)_\n$ of coinvariants of the action of $\n$ is a quotient of the $\Y(\gl_n)$-module $\Ec_{p,q}(V)$. Thus, we get a functor from the category of all $\gl_m$-modules to the category of bimodules over $\h$ and $\Y(\gl_n)$
\beq
\label{zelefun}
V\mapsto\Ec_{p,q}(V)_\n = (V\otimes\H)_\n.
\eeq

By the transitivity of induction, Theorem \ref{parind} can be extended from the maximal to all parabolic subalgebras of the Lie algebra $\gl_m$. Consider the case of the Borel subalgebra $\h\oplus\n'$ of $\gl_m$. Apply the functor \eqref{zelefun} to the $\gl_m$-module $V=M_\mu$, where $M_\mu$ is the Verma module of weight $\mu\in\h^\ast$. We obtain the $\hr{\h,\Y(\gl_n)}$-bimodule
$$
\Ec_{p,q}(M_\mu)_\n \,=\, (M_\mu\otimes\H)_\n.
$$
Using the basis $E_{11},\dots,E_{mm}$ we identify $\h$ with the direct sum of $m$ copies of the Lie algebra $\gl_1$. Consider the Verma modules $M_{\mu_1},\dots,M_{\mu_m}$ over $\gl_1$. By applying
Theorem~\ref{parind} repeatedly we get

\begin{cor}
\label{bimequiv}
The bimodule $\Ec_{p,q}(M_\mu)_\n$ over $\h$ and $\Y(\gl_n)$
is isomorphic to the tensor product
$$
\Ec_{1,0}(M_{\mu_1}) \otimes \Ec_{1,0}^1(M_{\mu_2}) \otimes \dots \otimes \Ec_{1,0}^{p-1}(M_{\mu_p}) \otimes \Ec_{0,1}^p(M_{\mu_{p+1}}) \otimes \dots \otimes \Ec_{0,1}^{m-1}(M_{\mu_m}).
$$
\end{cor}

\subsection*{\it\normalsize 2.2.\ Standard modules}

Let us now describe the bimodules $\Ec_{1,0}^z(M_t)$ and $\Ec_{0,1}^z(M_t)$ over $\gl_1$ and $\Y(\gl_n)$ for arbitrary $t,z\in\Cbb$. The Verma module $M_t$ over $\gl_1$ is one-dimensional, and the element $E_{11}\in\gl_1$ acts on $M_t$ by multiplication by $t$. The vector space of bimodules $\Ec_{1,0}(M_t)$ and $\Ec_{0,1}(M_t)$ is the algebra $\Hc(\Cbb^1\otimes\Cbb^n)=\Hc(\Cbb^n)$. Then $E_{11}$ acts on the bimodules $\Ec_{1,0}(M_t)$ and $\Ec_{0,1}(M_t)$ as differential operators
$$
t + \th\dfrac{n}2 - \th\suml{k=1}n\d_{1k}x_{1k}
\qquad \text{and} \qquad
t + \th\dfrac{n}2 + \suml{k=1}nx_{1k}\d_{1k}
$$
respectively.
The action of $E_{11}$ on $\Ec_{p,q}^z(M_t)$ is the same as on $\Ec_{p,q}(M_t)$. The action of $\Y(\gl_n)$ on $\Ec_{1,0}^z(M_t)$ and $\Ec_{0,1}^z(M_t)$ is given by
\beq \label{Tij}
T_{ij}(u)\mapsto\de_{ij}-\th\frac{\d_{1i}x_{1j}}{u+\th(t-z)}
\qquad \text{and} \qquad
T_{ij}(u)\mapsto\de_{ij}+\frac{x_{1i}\d_{1j}}{u+\th(t-z)}
\eeq
respectively, this is what Proposition \ref{prop_bimod} states in the case $m=1$. Note that both operators $x_{1i}\d_{1j}$ and $-\th\d_{1i}x_{1j}$ describe actions of the element $E_{ij}\in\gl_n$ on $\Hc(\Cbb^1\otimes\Cbb^n)=\Hc(\Cbb^n)$. When speaking about these actions we will omit the first indices and write $x_i$ and $\d_j$ instead of $x_{1i}$ and $\d_{1j}$. Hence, actions of the algebra $\Y(\gl_n)$ on $\Ec_{1,0}^z(M_t)$ and $\Ec_{0,1}^z(M_t)$ can be obtained from the actions of $\gl_n$ on $\Hc(\Cbb^n)$ by pulling back through the evaluation homomorphism $\pi_{\th(t-z)}$.

Now, consider $\Y(\gl_n)$-modules $\TPhi_z$ and $\Phi_z$ with the underlying vector space $\Hc(\Cbb^n)$ and Yangian actions defined by
$$
T_{ij}(u)\mapsto\de_{ij}-\th\frac{\d_ix_j}{u+\th z}
\qquad \text{and} \qquad
T_{ij}(u)\mapsto\de_{ij}+\frac{x_i\d_j}{u+\th z}
$$
correspondingly. Therefore, the bimodules $\Ec_{1,0}^z(M_t)$ and $\Ec_{0,1}^z(M_t)$ are respectively isomorphic to $\widetilde\Phi_{t-z}$ and $\Phi_{t-z}$ as the $\Y(\gl_n)$-modules. Moreover, corollary \ref{bimequiv} implies that the bimodule $\Ec_{p,q}(M_\mu)_\n\,$ of $\,\h\in\gl_m\,$ and $\,\Y(\gl_n)\,$ is isomorphic as a $\Y(\gl_n)$-module to the tensor product
\beq \label{tensor}
\TPhi_{\mu_1+\rho_1} \otimes \ldots \otimes \TPhi_{\mu_p+\rho_p}
\otimes
\Phi_{\mu_{p+1}+\rho_{p+1}} \otimes \ldots \otimes \Phi_{\mu_m+\rho_m}
\eeq
where $\rho_a = 1-a$.

Let us also define a $\Y(\gl_n)$-module $\Phi'_z$ with the underlying vector space $\Hc(\Cbb^n)$ and the $\Y(\gl_n)$-action given by
$$
T_{ij}(u) \mapsto \de_{ij}-\frac{x_j\d_i}{u+\th(z-1)}.
$$
Using commutation relation $\d_i x_j - \th x_j \d_i = \de_{ij}$ one can verify that
$$
\de_{ij}-\th\frac{\d_ix_j}{u+\th z} = \frac{u+\th(z-1)}{u+\th z}\hr{\de_{ij}-\frac{x_j\d_i}{u+\th(z-1)}},
$$
which implies the isomorphism of $\Y(\gl_n)$-modules
\pbeq
\begin{split}
\TPhi_z \cong \Om'_z \otimes \Phi'_z &\qquad\text{if}\qquad \th=1, \\
\TPhi_z \cong \Om_{-z} \otimes \Phi'_z &\qquad\text{if}\qquad \th=-1.
\end{split}
\peeq
Hence, the bimodule $\Ec_{p,q}(M_\mu)_\n$ over $\h\in\gl_m$ and $\Y(\gl_n)$ is isomorphic as the $\Y(\gl_n)$-module to the tensor product
\beq \label{standard}
\otimesl{a=1}{p} \Om^\ast_{\th(\mu_a+\rho_a)} \otimes
\otimesl{a=1}{p} \Phi'_{\mu_a+\rho_a} \otimes
\otimesl{a=p+1}{m} \Phi_{\mu_a+\rho_a}
\eeq
where
$$
\Om^\ast_z = \Om'_z \qquad\text{if}\quad \th = 1 \qquad\quad\text{and}\quad\qquad \Om^\ast_z = \Om_z \qquad\text{if}\quad \th = -1.
$$

Note that the $\Y(\gl_n)$-modules $\Phi_z$ and $\Phi'_z$ can also be realized as pull-backs of the $\gl_n$-modules $\Pc(\Cbb)$ and $\Gc(\Cbb)$ through the evaluation and the dual evaluation homomorphisms. Moreover, modules $\Phi_z$ and $\Phi'_z$ are rational, and hence so are their subquotients. We call an $\Y(\gl_n)$-module a \emph{standard rational module} if it is a tensor product of modules $\Phi_z$ and $\Phi'_z$ with arbitrary values of $z$.


\section*{\bf\normalsize 3.\ Zhelobenko operators}
\setcounter{section}{3}
\setcounter{theorem}{0}
\setcounter{equation}{0}

\subsection*{\it\normalsize 3.1.\ Definition}

Consider $\Sg_m$ as the Weyl group of the reductive Lie algebra $\gl_m$. Let $E_{11}^\ast, \dots, E_{mm}^\ast$ be the basis of $\h^\ast$ dual to the basis $E_{11}, \dots, E_{mm}$ of the Cartan subalgebra $\h\subset\gl_m$. The group $\Sg_m$ acts on the space $\h^\ast$ so that for any $\si\in\Sg_m$ and $a=1,\dots,m$
$$
\si \colon E_{aa}^\ast \mapsto E_{\si(a)\si(a)}^\ast.
$$
If we identify each weight $\mu\in\h^\ast$ with the sequence $(\mu_1, \dots, \mu_m)$ of its labels, then
$$
\si \colon (\mu_1, \dots, \mu_m) \mapsto (\mu_{\si^{-1}(1)}, \dots, \mu_{\si^{-1}(m)}).
$$
Let $\rho\in\h^\ast$ be the weight with sequence of labels $(0,-1,\dots,1-m)$. The \emph{shifted} action of any element $\si\in\Sg_m$ on $\h^\ast$ is defined by the assignment
$$
\mu \mapsto \si \circ \mu = \si(\mu + \rho) - \rho.
$$
The Weyl group also acts on the vector space $\gl_m$ so that for any $\si\in\Sg_m$ and $a,b=1,\dots,m$
$$
\si \colon E_{ab} \mapsto E_{\si(a)\si(b)}.
$$
The latter action extends to an action of the group $\Sg_m$ by automorphisms of the associative algebra $\U(\gl_m)$. The group $\Sg_m$ also acts by automorphisms of the space $\HD$ so that element $\si\in\Sg_m$ maps
$$
p_{ai} \mapsto p_{\si(a)i}  \qquad\text{and}\qquad  q_{ai} \mapsto q_{\si(a)i}.
$$
Note that homomorphisms \eqref{gl_emb} and \eqref{yang_hom} of the algebras $\gl_m$ and $\Y(\gl_n)$ into the algebra $\U(\gl_m)\otimes\HD$ are $\Sg_m$-equivariant.

Let $\A$ be the associative algebra generated by the algebras $\U(\gl_m)$ and $\HD$ with the cross relations
\beq
\label{defar}
[X,Y]=[\zeta_n(X),Y]
\eeq
for any $X\in\gl_m$ and $Y\in\HD$. The brackets at the left hand side of the relation \eqref{defar} denote the commutator in $\A$, while the brackets at the right hand side denote the commutator in the algebra $\HD$ embedded into $\A$. In particular, we will regard $\U(\gl_m)$ as a subalgebra of $\A$. An isomorphism of the algebra $\A$ with the tensor product $\U(\gl_m)\otimes\HD$ can be defined by mapping the elements $X\in\gl_m$ and $Y\in\HD$ in $\A$ respectively to the elements
$$
X\otimes1+1\otimes\zeta_n(X)
\quad\text{and}\quad
1\otimes Y
$$
in $\U(\gl_m)\otimes\HD$. The action of the group $\Sg_m$ on $\A$ is defined via the isomorphism of $\A$ with the tensor product $\U(\gl_m)\otimes\HD$. Since the homomorphism $\zeta_n$ is $\Sg_m$-equivariant the same action of $\Sg_m$ is obtained by extending the actions of $\Sg_m$ from the subalgebras $\U(\gl_m)$ and $\HD$ to $\A$.

For any $\,a,b = 1,\dots,m\,$ put $\,\eta_{ab}=E_{aa}^\ast-E_{bb}^\ast\in\h^\ast\,$ and $\,\eta_c = \eta_{c\,c+1}\,$ with $\,c = 1,\dots,m-1$. Put also
\beq\label{efh}
E_{c}=E_{c\,c+1}, \qquad F_{c}=E_{c+1\,c} \qquad\text{and}\qquad H_c = E_{cc}-E_{c+1\,c+1}.
\eeq
For any $c=1,\dots,m-1$ these three elements form an $\slg_2$-triple.

Let $\Uhb$ be the ring of fractions of the commutative algebra $\U(\h)$ relative to the set of denominators
\beq
\label{denset}
\hc{E_{aa}-E_{bb}+z\ |\ 1 \le a,b \le m;\ a \ne b;\  z\in\Zbb}.
\eeq
The elements of this ring can also be regarded as rational functions on the vector space $\h^\ast$. Then the elements of $\U(\h)\subset\Uhb$ become polynomial functions on $\h^\ast$. Denote by $\Ab$ the ring of fractions of $\A$ relative to the set of denominators \eqref{denset}, regarded as elements of $\A$ using the embedding of $\h\subset\gl_m$ into $\A$. The ring $\Ab$ is defined due to the following relations in the algebras $\U(\gl_m)$ and $\A$: for $a,b=1,\dots,m$ and $H\in\h$
\beq \label{Ore}
[H,E_{ab}] = \eta_{ab}(H)E_{ab},
\qquad
[H,p_{ak}] = -\th E_{aa}^\ast(H)p_{ak},
\qquad
[H,q_{bk}] = E_{bb}^\ast(H)q_{bk},
\eeq
where $p_{ak}$ and $q_{bk}$ are given by \eqref{new_coord}. Therefore, the ring $\A$ satisfies the Ore condition relative to its subset \eqref{denset}. Using left multiplication by elements of $\Uhb$, the ring of fractions $\Ab$ becomes a module over $\Uhb$.

The ring $\Ab$ is also an associative algebra over the field $\Cbb$. The action of the group $\Sg_m$ on $\A$ preserves the set of denominators \eqref{denset} so that $\Sg_m$ also acts by automorphisms of the algebra $\Ab$. For each $c=1,\dots,m-1$ define a linear map $\xi_c \colon \A\to\Ab$ by setting
\beq
\label{q1}
\xi_c(Y) = Y+\suml{s=1}\infty(s!H_c^{(s)})^{-1}E_c^s \widehat{F}_c^s(Y)
\eeq
for any $Y\in\A$. Here
$$
H_c^{(s)}=H_c(H_c-1)\dots(H_c-s+1)
$$
and $\widehat{F}_c$ is the operator of adjoint action corresponding to the element $F_c\in\A$, so that
$$
\widehat{F}_c(Y)=[F_c,Y].
$$
For any given element $Y\in\A$ only finitely many terms of the sum \eqref{q1} differ from zero, hence the map $\xi_c$ is well defined.

Let $\J$ and $\Jb$ be the right ideals of the algebras $\A$ and $\Ab$ respectively, generated by all elements of the subalgebra $\n\subset\gl_m$. Let $\Jp$ be the left ideal of the algebra $\A$, generated by the elements $X-\zeta_n(X)$, or equivalently by the elements
\beq \label{differ}
X\otimes1\in\U(\gl_m)\otimes1 \subset \U(\gl_m)\otimes\HD,
\eeq
where $X\in\np$. Denote $\Jpb=\Uhb\Jp$, then $\Jpb$ is a left ideal of the algebra $\Ab$.

Now we give a short observation of some results proved in \cite{KN1}. For any elements $X\in\h$ and $Y\in\A$ we have
$$
\xi_a(X Y) \in (X+\eta_a(X))\xi_a(Y)+\Jb.
$$
This allows us to define a linear map $\bar\xi_a\colon\Ab\to\Jb\,\backslash\,\Ab$
by setting
\beq \label{N14}
\bar\xi_a(XY)=Z\xi_a(Y)+\Jb
\quad\text{for}\quad
X\in\Uhb
\quad\text{and}\quad
Y\in\A
\eeq
where the element $Z\in\Uhb$ is defined by the equality
\beq \label{N1414}
Z(\mu)=X(\mu+\eta_a)
\quad\text{for}\quad
\mu\in\h^\ast
\eeq
and both $X$ and $Z$ are regarded as rational functions on $\h^\ast$. The backslash in $\Jb\,\backslash\,\Ab$ indicates that the quotient is taken relative to a \emph{right} ideal of $\Ab$.

The action of the group $\Sg_m$ on the algebra $\U(\gl_m)$ extends to an action on $\Uhb$ so that
for any $\si\in\Sg_m$
$$
(\si(X))(\mu)=X(\si^{-1}(\mu)),
$$
when the element $X\in\Uhb$ is regarded as a rational function on $\h^\ast$. The action of $\Sg_m$ by automorphisms of the algebra $\A$ then extends to an action by automorphisms of $\Ab$. For any $c=1,\dots, m-1\,$ let $\,\si_c\in\Sg_m\,$ be the transposition of $c$ and $c+1$. Consider the image $\si_c(\Jb)$, that is again a right ideal of $\Ab$. By \cite[Proposition 3.2]{KN1} we have
$$
\si_c(\Jb)\subset\ker\bar\xi_c.
$$

This allows us to define for any $c=1,\dots, m-1$ a linear map
\beq
\label{xic}
\xic_c\colon\Jb\,\backslash\,\Ab\to\Jb\,\backslash\,\Ab
\eeq
as the composition $\bar\xi_c\si_c$ applied to the elements of $\Ab$ taken modulo $\Jb$. The operators $\xic_1,\dots,\xic_{m-1}$ on the vector space $\Jb\,\backslash\,\Ab$ are called \emph{the Zhelobenko operators}. By \cite[Proposition 3.3]{KN1} the Zhelobenko operators $\xic_1,\dots,\xic_{m-1}$ on $\Jb\,\backslash\,\Ab$ satisfy the braid relations
\begin{align*}
\hspace{25pt} \xic_c\,\xic_{c+1}\,\xic_c &\,=\, \xic_{c+1}\,\xic_c\,\xic_{c+1} &\hspace{-50pt} &\text{for}\quad c<m-1, \\
\hspace{25pt} \xic_b\,\xic_c &\,=\, \xic_c\,\xic_b &\hspace{-50pt} &\text{for}\quad |b-c|>1.
\end{align*}
Therefore, for any reduced decomposition $\si=\si_{c_1}\dots\si_{c_k}$ in $\Sg_m$ the composition $\xic_{c_1}\dots\xic_{c_k}$ of operators on $\Jb\,\backslash\,\Ab$ does not depend on the choice of decomposition of $\si$. Finally, for any $\si\in\Sg_m$, $X\in\Uhb$, and $Y\in\Jb\,\backslash\,\Ab$ we have relations
\beq \label{XY}
\xic_{\si}(XY) = (\si\circ X)\xic_{\si}(Y)
\eeq
\beq \label{YX}
\xic_{\si}(YX) = \xic_{\si}(Y)(\si\circ X)
\eeq
which follow from \cite[Proposition 3.1]{KN1}.

\subsection*{\it\normalsize 3.2.\ Intertwining properties}

Let $\de=(\de_1,\dots,\de_m)$ be a sequence of $m$ elements from the set $\hc{1,-1}$. Let the symmetric group $\Sg_m$ act on the set of sequences $\hc{\de}$ by
$$
\si(\de) = \ov{\si\cdot\ov\de}
$$
where
$$
\ov\de = (\de_1,\dots,\de_p,-\de_{p+1},\dots,-\de_{p+q})
\qquad\text{and}\qquad
\si\cdot\de = (\de_{\si^{-1}(1)},\dots,\de_{\si^{-1}(m)}).
$$
Denote
$$
\de^+ = (1,\dots,1)
\qquad\text{and}\qquad
\de' = \ov{\de^+} = (\underbrace{1,\dots,1}_p,\underbrace{-1,\dots,-1}_q).
$$
For a given sequence $\de$ let $\varpi$ denote a composition of automorphisms of the ring $\HD$ such that
\beq \label{varpi}
x_{ak} \mapsto -\th\d_{ak} \quad\text{and}\quad \d_{ak} \mapsto x_{ak}
\qquad\text{whenever}\qquad \de_a = -1.
\eeq
For any $\gl_m$-module $V$ we define a bimodule $\Ec_\de(V)$ over $\gl_m$ and $\Y(\gl_n)$. Its underlying vector space is $V\otimes\H$ for every $\de$. The action of the algebra $\gl_m$ on the module $\Ec_\de(V)$ is defined by pushing the homomorphism $\xi_n$ forward through the automorphism $\varpi$
$$
E_{ab} \;\mapsto\; E_{ab} \otimes 1 + 1 \otimes \varpi(\xi_n(E_{ab})).
$$
The action of the algebra $\Y(\gl_n)$ on the module $\Ec_\de(V)$ is defined by pushing the homomorphism $\al_n$ forward through the automorphism $\varpi$, applied to the second tensor factor $\HD$ of the target of $\al_n$
$$
T_{ij}(u) \;\mapsto\; \de_{ij} + \suml{a,b=1}{m} X_{ab}(u)\otimes\varpi(\Eh_{ai,bj}).
$$
For instance, we have $\Ec_{p,q}(V) = \Ec_{\de^+}(V)$.

Let $\mu\in\h^\ast$ be a generic weight of $\gl_m$, which means that
\beq \label{notinso}
\mu_a-\mu_b\not\in\Zbb \qquad\text{for all}\qquad a,b=1,\dots,m.
\eeq
In the remaining of this section we show that the Zhelobenko operator $\xic_\si$ determines an intertwining operator
\beq \label{distoper}
\Ec_{p,q}(M_\mu)_\n \to \Ec_\de(M_{\si\circ\mu})_\n \qquad\text{where}\qquad \de=\si(\de^+).
\eeq

Let $\I_\de$ be the left ideal of the algebra $\A$ generated by the elements $x_{ak},\;k=1,\dots,n,\,$ for $\de_a=-1$ and by the elements $\d_{ak},\;k=1,\dots,n,\,$ for $\de_a=1$. For instance, ideal $\I_{\de^+}$ is generated by the elements $\d_{ak}$ with $a=1,\dots,m$ and $k=1,\dots,n$. Let $\Ib_\de$ be the left ideal of the algebra $\Ab$ generated by the same elements as the ideal $\I_\de$ in $\A$. Occasionally, $\Ib_\de$ will denote the image of the ideal $\Ib_\de$ in the quotient space $\Jb\,\backslash\,\Ab$.

\begin{prop}
For any $\si\in\Sg_m$ the operator $\xic_\si$ maps the subspace $\Ib_{\de^+}$ to $\Ib_{\si(\de^+)}$.
\end{prop}

\begin{proof}
For all $a=1,\dots,m-1$ consider the operator $\widehat{F}_a$. Due to \eqref{defar}, \eqref{efh}, and \eqref{zeta} we have
$$
\widehat F_a(Y) = \suml{k=1}n\hs{q_{a+1\,k}p_{ak},Y}
$$
for any $Y\in\HD$. The above description of the action of $\widehat{F}_a$ with $a=1,\dots,m-1$ on the vector space $\HD$ shows that this action preserves each of the two $2n$ dimensional subspaces spanned by the vectors
\begin{align} \label{qqideal}
q_{ai} &\quad\text{and}\quad q_{a+1\,i} \qquad\text{where}\qquad i=1,\dots,n;
\\ \label{ppideal}
p_{ai} &\quad\text{and}\quad p_{a+1\,i} \qquad\text{where}\qquad i=1,\dots,n.
\end{align}
This action also maps to zero the $2n$ dimensional subspace spanned by
\beq \label{pqideal}
p_{ai}  \quad\text{and}\quad q_{a+1\,i} \qquad\text{where}\qquad i=1,\dots,n.
\eeq
Therefore, for any $\de$ the operator $\bar\xi_a$ maps the ideal $\Ib_\de$ of $\Ab$ to the image of $\Ib_\de$ in $\Jb\,\backslash\,\Ab$ unless $\de_a=\de'_a$ and $\de_{a+1}=-\de'_{a+1}$ for $a=1,\dots,m-1$. Hence, the operator $\xic_a = \bar\xi_a\si_a$ maps the subspace $\Ib_\de$ to the image of $\Ib_{\si(\de)}$ unless $\de_a=-\de'_a$ and $\de_{a+1}=\de'_{a+1}$.

From now on we will denote the image of the ideal $\Ib_\de$ in the quotient space $\Jb\,\backslash\,\Ab$ by the same symbol $\Ib_\de$. Put
$$
\widehat{\de}=\suml{a=1}{p}\de_a E_{aa}^\ast - \suml{a=p+1}{p+q}\de_a E_{aa}^\ast.
$$
Then for every $\si\in\Sg_m$ we have the equality $\widehat{\si(\de)}=\si(\widehat{\de})$ where at the right hand side we use the action of the group $\Sg_m$ on $\h^\ast$. Let $(\ ,\,)$ be the standard bilinear form on $\h^\ast$ so that the basis of weights $E_{aa}^\ast$ with $a=1,\dots,m$ is orthonormal. The above remarks on the action of the Zhelobenko operators on $\Ib_\de$ can now be rewritten as
\beq \label{root_cond}
\xic_a(\Ib_\de) \subset \Ib_{\si_a(\de)} \qquad\text{if}\qquad (\,\widehat{\de}\,,\,\eta_a) \ge 0.
\eeq

We will prove the proposition by induction on the length of the reduced decomposition of $\si$. Recall that the length $\ell(\si)$ of a reduced decomposition of $\si$ is the total number of the factors $\si_1,\dots,\si_{m-1}$ in that decomposition. This number is independent of the choice of decomposition and is equal to the number of elements in the set
$$
\De_\si = \hc{\eta\in\De^+\,|\,\si(\eta)\not\in\De^+}
$$
where $\De^+$ denotes the set of all positive roots of the Lie algebra $\gl_m$.

If $\si$ is the identity element of $\Sg_m$ then the statement of the proposition is trivial. Suppose that for some $\si\in\Sg_m$
$$
\xic_{\si}\hr{\Ib_{\de^+}} \subset \Ib_{\si(\de^+)}.
$$
Take $\si_a$ such that
\beq \label{length}
\ell(\si_a\si) = \ell(\si)+1.
\eeq
Now it is only left to prove that
$$
\xic_a\hr{\Ib_{\si(\de^+)}} \subset \Ib_{\si_a\si(\de)}.
$$
Due to \eqref{root_cond}, the desired property will take place if
$$
(\,\widehat{\si(\de^+)}\,,\,\eta_a\,) \,=\, (\,\si(\widehat{\de^+})\,,\,\eta_a\,) \,\ge\, 0.
$$

Note that $\eta_a$ is a simple root of the algebra $\gl_m$ and hence, $\si_a(\eta)\in\De^+$ for any $\eta\in\De^+$ such that $\eta\ne\eta_a$. Since $\ell(\si)$ and $\ell(\si_a\si)$ are the numbers of elements in $\De_\si$ and $\De_{\si_a\si}$ respectively, condition \eqref{length} implies that $\eta_a\in\si(\De^+)$. Therefore, $\eta_a = \si(E_{bb}^\ast - E_{cc}^\ast)$ for some $1 \le b<c \le m$. Thus
$$
\hr{\,\si(\widehat{\de^+})\,,\,\eta_a\,} =
\hr{\,\si(\widehat{\de^+})\,,\,\si(E_{bb}^\ast - E_{cc}^\ast)\,} =
\hr{\,\suml{a=1}{m}\de'_a E_{aa}^\ast\,,\,E_{bb}^\ast - E_{cc}^\ast} \ge 0.
$$
\end{proof}

Then following \cite[Corollary 5.2]{KN3} we obtain
\begin{cor}\label{JIJ}
For any $\si\in\Sg_m$ the operator $\xic_\si$ on $\Jb\,\backslash\,\Ab$ maps
$$
\Jb\,\backslash\,(\Jpb+\Ib_{\de^+}+\Jb)
\,\to\,
\Jb\,\backslash\,(\Jpb+\Ib_{\si(\de^+)}+\Jb).
$$
\end{cor}

For a generic weight $\mu$ let $\I_{\mu,\de}$ be the left ideal of the algebra $\A$ generated by $\I_{\de}+\Jp$ and by the elements
$$
E_{aa}-\zeta_n(E_{aa})-\mu_a
\quad\text{where}\quad
a=1,\dots,m.
$$
Recall that under the isomorphism of the algebra $\A$ with $\U(\gl_m)\otimes\HD$ element $X-\zeta_n(X)\in\A$ maps to the element \eqref{differ} for every $X\in\gl_m$. Let $\Ib_{\mu,\de}$ denote the subspace $\Uhb\,\I_{\mu,\de}$ of $\Ab$. Note that $\Ib_{\mu,\de}$ is a left ideal of the algebra $\Ab$.

\begin{theorem}
\label{proposition4.5}
For any element $\si\in\Sg_m$ the operator $\xic_\si$ on $\Jb\,\backslash\,\Ab$ maps
$$
\Jb\,\backslash\,(\Ib_{\mu,\de^+}+\Jb)
\,\to\,
\Jb\,\backslash\,(\Ib_{\si\circ\mu,\si(\de^+)}+\Jb).
$$
\end{theorem}

\begin{proof}
Let $\ka$ be a weight of $\gl_m$ with the sequence of labels $(\ka_1,\dots,\ka_m)$. Suppose that the weight $\ka$ satisfies the conditions \eqref{notinso} instead of $\mu$. Let $\tilde\I_{\ka,\de}$ denote the left ideal of $\Ab$ generated by $\I_\de+\Jp$ and by the elements
$$
E_{aa}-\ka_a \quad\text{where}\quad a=1,\dots,m.
$$
Relation \eqref{YX} and Corollary \ref{JIJ} imply that the operator $\xic_\si$ on $\Jb\,\backslash\,\A$ maps
$$
\Jb\,\backslash\,(\tilde\I_{\ka,\de^+}+\Jb) \,\to\,
\Jb\,\backslash\,(\tilde\I_{\si\circ\ka,\si(\de^+)}+\Jb).
$$
Now choose
\beq
\label{kappa}
\ka=\mu-\th\dfrac{n}2\de'
\eeq
where the sequence $\de'$ is regarded as a weight of $\gl_m$ by identifying the weights with their sequence of labels. Then the conditions on $\ka$ stated in the beginning of this proof are satisfied. For every $\si\in\Sg_m$ we shall prove the equality
of left ideals of $\Ab$,
\beq
\label{lastin}
\tilde\I_{\si\circ\ka,\si(\de^+)} = \Ib_{\si\circ\mu,\si(\de^+)}.
\eeq
Theorem \ref{proposition4.5} will then follow. Denote $\de=\si\cdot\de'$. Then by our choice of $\ka$, we have
\beq
\label{sika}
\si\circ\ka=\si\circ\mu-\th\dfrac{n}2\de.
\eeq

Let the index $a$ run through $1,\dots,m-1$, then
\begin{align*}
\zeta_n(E_{aa})+\th\dfrac{n}2 = \th\suml{k=1}{n}p_{ak}q_{ak}\in\I_{\si(\de^+)}
&\qquad\text{if}\qquad \de_a=1, \\
\zeta_n(E_{aa})-\th\dfrac{n}2 = \phantom{\th}\suml{k=1}{n}q_{ak}p_{ak}\in\I_{\si(\de^+)}
&\qquad\text{if}\qquad \de_a=-1.
\end{align*}
Hence, the relation \eqref{sika} implies the equality \eqref{lastin}.
\end{proof}

Consider the quotient vector space $\A\,/\,\I_{\mu,\de}$ for any sequence $\de$. The algebra $\U(\gl_m)$ acts on this quotient via left multiplication, being regarded as a subalgebra of $\A$. The algebra $\Y(\gl_n)$ also acts on this quotient via left multiplication, using the homomorphism $\al_m:\Y(\gl_n)\to\A$. Recall that in Section~1 the target algebra of the homomorphism $\al_m$
was defined as the tensor product $\U(\gl_m)\otimes\HD$ isomorphic to the algebra $\A$ by means of the cross relations~\eqref{defar}. Part~(ii) of the proposition \ref{prop_bimod} implies that the image of $\al_m$ in $\A$ commutes with the subalgebra $\U(\gl_m)\subs\A$. Thus, the vector space $\A\,/\,\I_{\mu,\de}$ becomes a bimodule over $\gl_m$ and $\Y(\gl_n)$.

Consider the bimodule $\Ec_\de(M_\mu)$ over $\gl_m$ and $\Y(\gl_n)$ defined in the beginning of this section. This bimodule is isomorphic to $\A\,/\,\I_{\mu,\de}$. Indeed, let $Z$ run through $\H$. Then a bijective linear map
$$
\Ec_\de(M_\mu)\to\A\,/\,\I_{\mu,\de},
$$
intertwining the actions of $\gl_m$ and $\Y(\gl_n)$, can be determined by mapping the element
$$
1_\mu\otimes Z \;\in\; M_\mu\otimes\H
$$
to the image of
$$
\varpi^{-1}(Z)\,\in\,\HD\,\subset\,\A
$$
in the quotient $\A\,/\,\I_{\mu,\de}$, where $\varpi$ is defined by \eqref{varpi}. The intertwining property here follows from the definitions of $\Ec_\de(M_\mu)$ and $\I_{\mu,\de}$. The same mapping determines a bijective linear map
\beq \label{fbi}
\Ec_\de(M_\mu)\,\to\,\Ab\,/\,\Ib_{\mu,\de}.
\eeq

In particular, the space $\Ec_\de(M_\mu)_\n$ of $\n$-coinvariants of $\Ec_\de(M_\mu)$ is isomorphic to the quotient $\Jb\,\backslash\,\Ab\,/\,\Ib_{\mu,\de}$ as a bimodule over the Cartan subalgebra $\h\subset\gl_m$ and over $\Y(\gl_n)$. Theorem \ref{proposition4.5} implies that the operator $\xic_\si$ on $\Jb\,\backslash\,\Ab$ determines a linear operator
\beq \label{bbjioper}
\Jb\,\backslash\,\Ab\,/\,\Ib_{\mu,\de^+} \to\,
\Jb\,\backslash\,\Ab\,/\,\Ib_{\si\circ\mu,\si(\de^+)}.
\eeq
The definition \eqref{q1} and the fact that the image of $\Y(\gl_n)$ in $\A$ under $\al_m$ commutes with the subalgebra $\U(\gl_m)\subset\A$ imply that the latter operator intertwines the actions of $\Y(\gl_n)$ on the source and the target vector spaces. We also use the invariance of the image of $\Y(\gl_n)$ in $\A$ under the action of $\Sg_m$. Recall that $\Ec_{p,q}(V)=\Ec_{\de^+}(V)$. Hence, by using the equivalences \eqref{fbi} for the sequences $\de=\de^+$ and $\de=\si(\de^+)$, the operator \eqref{bbjioper} becomes the desired $\Y(\gl_n)$-intertwining operator \eqref{distoper}.

As usual, for any $\gl_m$-module $V$ and any element $\la\in\h^\ast$ let $V^\la\subset V$ be the subspace of vectors \emph{of weight} $\la$ so that any $X\in\h$ acts on $V^\la$ via multiplication by $\la(X)\in\Cbb$. It now follows from the property \eqref{XY} of $\xic_\si$ that the restriction of our operator \eqref{distoper} to the subspace of weight $\la$ is an $\Y(\gl_n)$-intertwining operator
\beq \label{distoperla}
\Ec_{p,q}(M_\mu)_\n^\la \,\to\, \Ec_\de(M_{\si\circ\mu})_\n^{\si\circ\la}
\quad\text{where}\quad\ \de=\si(\de^+).
\eeq

Consider $\Y(\gl_n)$-modules $\Phi_z$ and $\TPhi_z$ described at the end of the Section~2. The underlying vector space of each of these modules coincides with the algebra $\Hc(\Cbb^n)$. Note that the action of $\Y(\gl_n)$ on each of these modules preserves the polynomial degree. Now for any $N=1,2,\dots$ denote respectively by $\Psi_z^N$ and $\Psi_z^{-N}$ the submodules in $\Phi_z$ and $\TPhi_z$ consisting of the polynomial functions of degree $N$. It will also be convenient to denote by $\Psi_z^0$ the vector space $\Cbb$ with the trivial action of $\Y(\gl_n)$.

The element $E_{11}\in\gl_1$ acts on $\Ec_{1,0}(M_t)$ and $\Ec_{0,1}(M_t)$ by
$$
t-\th\frac{n}2-\deg   \qquad\text{and}\qquad   t+\th\frac{n}2+\deg
$$
respectively where $\deg$ is the degree operator. Therefore, corollary \ref{bimequiv} yields the isomorphism between the source $\Y(\gl_n)$-module in \eqref{distoperla} and the tensor product
\beq \label{source}
\Psi_{\mu_1+\rho_1}^{-\nu_1} \otimes \dots \otimes \Psi_{\mu_p+\rho_p}^{-\nu_p} \otimes
\Psi_{\mu_{p+1}+\rho_{p+1}}^{\nu_{p+1}} \otimes \dots \otimes \Psi_{\mu_m+\rho_m}^{\nu_m}
\eeq
where
\beq \label{nu}
\begin{split}
\nu_a =          - \la_a+\mu_a-\th\dfrac{n}2 &\qquad\text{for}\qquad a=  1,\dots,p, \\
\nu_a = \phantom{-}\la_a-\mu_a-\th\dfrac{n}2 &\qquad\text{for}\qquad a=p+1,\dots,m.
\end{split}
\eeq

Let us now consider the target $\Y(\gl_n)$-module in \eqref{distoperla}. For each $a=1,\dots,m$ denote
$$
\widetilde\mu_a=\mu_{\si^{-1}(a)},
\quad
\widetilde\nu_a=\nu_{\si^{-1}(a)},
\quad
\widetilde\rho_a=\rho_{\si^{-1}(a)}.
$$
The above description of the source $\Y(\gl_n)$-module in \eqref{distoperla} can now be generalized to the target $\Y(\gl_n)$-module which depends on an arbitrary element $\si\in\Sg_m$.

\begin{prop} \label{siverma}
For $\de=\si(\de^+)$ the $\Y(\gl_n)$-module $\Ec_\de(M_{\si\circ\mu})_\n^{\si\circ\la}$ is isomorphic to the tensor product
\beq \label{ppd}
\Psi_{\widetilde\mu_1+\widetilde\rho_1}^{-\de_1\widetilde\nu_1}
\otimes \dots \otimes
\Psi_{\widetilde\mu_p+\widetilde\rho_p}^{-\de_p\widetilde\nu_p}
\otimes
\Psi_{\widetilde\mu_{p+1}+\widetilde\rho_{p+1}}^{\de_{p+1}\widetilde\nu_{p+1}}
\otimes \dots \otimes
\Psi_{\widetilde\mu_m+\widetilde\rho_m}^{\de_m\widetilde\nu_m}.
\eeq
\end{prop}

\begin{proof}
Consider the bimodule $\Ec_{p,q}(M_{\si\circ\mu})_\n^{\si\circ\la}$ over $\h$ and $\Y(\gl_n)$. By Corollary \ref{bimequiv} and the arguments just above this proposition, this bimodule is isomorphic to the tensor product
\beq \label{bim1}
\Psi_{\widetilde\mu_1+\widetilde\rho_1}^{-\widetilde\nu_1}
\otimes \dots \otimes
\Psi_{\widetilde\mu_p+\widetilde\rho_p}^{-\widetilde\nu_p}
\otimes
\Psi_{\widetilde\mu_{p+1}+\widetilde\rho_{p+1}}^{\widetilde\nu_{p+1}}
\otimes \dots \otimes
\Psi_{\widetilde\mu_m+\widetilde\rho_m}^{\widetilde\nu_m}
\eeq
as a $\Y(\gl_m)$-module. The bimodule $\Ec_\de(M_{\si\circ\mu})_\n$ can be obtained by pushing forward the actions of $\h$ and $\Y(\gl_n)$ on $\Ec_{p,q}(M_{\si\circ\mu})_\n$ through the composition of automorphisms
\beq \label{onefour}
x_a \to -\th\d_a \qquad\text{and}\qquad \d_a \to x_a
\eeq
for every tensor factor with number $a$ such that $\de_a=-1$. These automorphisms exchange $\Y(\gl_n)$-modules $\Phi_{z_a}$ and $\TPhi_{z_a}$ but leave invariant the degree of polynomials. Therefore, they interchange $\Psi_{z_a}^N$ and $\Psi_{z_a}^{-N}$ which implies the resulting $\Y(\gl_n)$-module to be as in \eqref{ppd}.
\end{proof}

Thus, for any non-negative integers $\nu_1,\dots,\nu_m$ we have shown that the Zhelobenko operator $\xic_\si$ on $\Jb\,\backslash\,\Ab$ defines the intertwining operator between the $\Y(\gl_n)$-modules
$$
\Psi_{\mu_1+\rho_1}^{-\nu_1} \otimes \dots \otimes \Psi_{\mu_p+\rho_p}^{-\nu_p} \otimes
\Psi_{\mu_{p+1}+\rho_{p+1}}^{\nu_{p+1}} \otimes \dots \otimes \Psi_{\mu_m+\rho_m}^{\nu_m}
$$
and
$$
\Psi_{\widetilde\mu_1+\widetilde\rho_1}^{-\de_1\widetilde\nu_1}
\otimes \dots \otimes
\Psi_{\widetilde\mu_p+\widetilde\rho_p}^{-\de_p\widetilde\nu_p}
\otimes
\Psi_{\widetilde\mu_{p+1}+\widetilde\rho_{p+1}}^{\de_{p+1}\widetilde\nu_{p+1}}
\otimes \dots \otimes
\Psi_{\widetilde\mu_m+\widetilde\rho_m}^{\de_m\widetilde\nu_m}.
$$
Moreover, the operator $\xic_\si$ permutes tensor factors of the module \eqref{source}, therefore modules $\Ec_{p,q}(M_\mu)_\n^\la$ and $\Ec_\de(M_{\si\circ\mu})_\n^{\si\circ\la}$, written in the form \eqref{standard}, contain similar tensor factors $\Om_{z_a}$ or $\Om'_{z_a}$. Now recall that modules $\Om_{z_a}$ and $\Om'_{z_a}$ are cocentral and that the Zhelobenko operator $\xic_\si$ acts on them as identity operator. These two observations allow us to exclude $\Om_{z_a}$ and $\Om'_{z_a}$ from both source and target modules of the mapping \eqref{distoperla} and to define an intertwining operator $\xi'_\si$ between tensor products of modules $\Phi_z$ and $\Phi'_z$. For any integer $N$ denote respectively by $\Phi_z^N$ and $\Phi_z^{-N}$ the submodules of $\Phi_z$ and $\Phi'_z$ consisting of polynomial functions of degree $N$. Note that $\Phi_z^N$ coincides with $\Psi_z^N$ for positive $N$ but differs for negative $N$. Hence, the following theorem holds.

\begin{theorem}
\label{th_inter}
Given a generic weight $\mu$ and non-negative integers $\nu_1,\dots,\nu_m$, the map $\xi'_\si$ intertwines rational $\Y(\gl_n)$-modules
\beq \label{rational_source}
\Phi_{\mu_1+\rho_1}^{-\nu_1}
\otimes \dots \otimes
\Phi_{\mu_p+\rho_p}^{-\nu_p}
\otimes
\Phi_{\mu_{p+1}+\rho_{p+1}}^{\nu_{p+1}}
\otimes \dots \otimes
\Phi_{\mu_m+\rho_m}^{\nu_m}
\eeq
and
$$
\Phi_{\widetilde\mu_1+\widetilde\rho_1}^{-\de_1\widetilde\nu_1}
\otimes \dots \otimes
\Phi_{\widetilde\mu_p+\widetilde\rho_p}^{-\de_p\widetilde\nu_p}
\otimes
\Phi_{\widetilde\mu_{p+1}+\widetilde\rho_{p+1}}^{\de_{p+1}\widetilde\nu_{p+1}}
\otimes \dots \otimes
\Phi_{\widetilde\mu_m+\widetilde\rho_m}^{\de_m\widetilde\nu_m}.
$$
\end{theorem}

\section*{\bf\normalsize 4.\ Highest weight vectors}
\setcounter{section}{4}
\setcounter{theorem}{0}
\setcounter{equation}{0}

\subsection*{\it\normalsize 4.1.\ Symmetric case}

In this subsection we consider only the case of commuting variables, hence from now on and till the end of the subsection we assume $\th=1$. Proposition \ref{isiscomm} below determines the image of the highest weight vector $v_\mu^\la$ of the $\Y(\gl_n)$-module $\Ec_{p,q}(M_\mu)_\n^\la$ under the action of the operator $\xic_\si$. The proof of the Proposition~\ref{isiscomm} is based on the following three lemmas. Proofs of the first two of them are similar to the proof of Lemma 5.6 in \cite{KN3}. Proof of the last one is similar to the proof of Lemma 5.7 in \cite{KN3}. Let $s,t=0,1,2\dots$ and $k,\ell=1,\dots,n$.

\begin{lem}
\label{ppnorma}
For any\/ $a=1,\dots,m-1$ the operator\/ $\xic_a$ on\/ $\Jb\,\backslash\,\Ab$ maps the image in\/ $\Jb\,\backslash\,\Ab$ of\/ $p_{ak}^s p_{a+1\,k}^t\in\Ab$ to the image in\/ $\Jb\,\backslash\,\Ab$ of the product
$$
p_{ak}^t p_{a+1\,k}^s\cdot
\prodl{r=1}s\,\frac{H_a+r+1}{H_a+r-t}
$$
plus the images in\/ $\Jb\,\backslash\,\Ab$ of elements of the left ideal in\/ $\Ab$ generated by\/ $\Jpb$ and \eqref{qqideal}.
\end{lem}

\begin{lem}
\label{qqnorma}
For any\/ $a=1,\dots,m-1$ the operator\/ $\xic_a$ on\/ $\Jb\,\backslash\,\Ab$ maps the image in\/ $\Jb\,\backslash\,\Ab$ of\/ $q_{ak}^s q_{a+1\,k}^t\in\Ab$ to the image in\/ $\Jb\,\backslash\,\Ab$ of the product
$$
q_{ak}^t q_{a+1\,k}^s\cdot
\prodl{r=1}t\,\frac{H_a+r+1}{H_a+r-s}
$$
plus the images in\/ $\Jb\,\backslash\,\Ab$ of elements of the left ideal in\/ $\Ab$ generated by\/ $\Jpb$ and \eqref{ppideal}.
\end{lem}

\begin{lem}
\label{pqnorma}
For any\/ $a=1,\dots,m-1$ the operator\/ $\xic_a$ on\/ $\Jb\,\backslash\,\Ab$ maps the image in\/ $\Jb\,\backslash\,\Ab$ of\/ $p_{ak}^s q_{a+1\,\ell}^t\in\Ab$ to the image in\/ $\Jb\,\backslash\,\Ab$ of the product
$$
q_{a\ell}^t p_{a+1\,k}^s\cdot
\left\{
\begin{aligned}
\prodl{r=1}s\,\frac{H_a+r}{H_a+r+t} & \qquad\text{if}\quad n=1 \quad\text{and}\quad k=\ell=1, \\
1\hspace{30pt} &\qquad\text{if}\quad n>1 \quad\text{and}\quad k\ne\ell
\end{aligned}
\right.
$$
plus the images in\/ $\Jb\,\backslash\,\Ab$ of elements of the left ideal in\/ $\Ab$ generated by\/ $\Jpb$ and \eqref{pqideal}.
\end{lem}

We keep assuming that the weight $\mu$ satisfies the condition \eqref{notinso}. Let $(\mu^\ast_1,\dots,\mu^\ast_m)$ be the sequence of labels of weight $\mu+\rho-\frac{n}2\de'$. Suppose that for all $a=1,\dots,m$ the number $\nu_a$ defined by \eqref{nu} is a non-negative integer. For each positive root $\eta=E_{bb}^\ast-E_{cc}^\ast\in\Delta^+$ with $1\le b<c \le m$ define a number $z_\eta\in\Cbb$ by
$$
z_\eta =
\left\{
\begin{array}{llll}
\displaystyle
\prodl{r=1}{\nu_b}\frac{\mu^\ast_b-\mu^\ast_c-r}{\la^\ast_b-\la^\ast_c+r}
&\quad\text{if}\qquad
b,c = 1,\dots,p, \\[12pt]
\displaystyle
\prodl{r=1}{\nu_c}\frac{\mu^\ast_b-\mu^\ast_c-r}{\la^\ast_b-\la^\ast_c+r}
&\quad\text{if}\qquad
b,c = p+1,\dots,m, \\[12pt]
\displaystyle
\prodl{r=1}{\upsilon}\frac{\mu^\ast_b-\mu^\ast_c-r+1}{\la^\ast_b-\la^\ast_c+r-1}
&\quad\text{if}\qquad
\hspace{-8.5pt} \mbox{ $\begin{array}{rcl}
\displaystyle
b&=&1,\dots,p, \\
c&=&p+1,\dots,m,
\end{array}$} & \quad\text{and}\quad n=1, \\[12pt]
1, &\quad\text{if}\qquad
\hspace{-8.5pt} \mbox{ $\begin{array}{rcl}
\displaystyle
b&=&1,\dots,p, \\
c&=&p+1,\dots,m,
\end{array}$} & \quad\text{and}\quad n>1.
\end{array}
\right.
$$
Here $\upsilon=\min(\nu_b,\nu_c)$.
Let $v_{\mu}^\la$ denote the image of the vector
$$
\prodl{a=1}{p}\,p_{an}^{\nu_a} \;\cdot\! \prodl{a=p+1}{m}\!\!q_{a1}^{\nu_a} \,\in\, \Ab
$$
in the quotient space $\Jb\,\backslash\,\Ab/\,\Ib_{\mu,\de^+}$.

\begin{prop} \label{isiscomm}
{\rm\,\,(i)}
The vector $v_{\mu}^\la$ does not belong to the zero coset in $\Jb\,\backslash\,\Ab/\,\Ib_{\mu,\de^+}$. \\
{\rm\,\,(ii)}
The vector $v_{\mu}^\la$ is of weight $\lambda$ under the action of\/ $\h$ on\/
$\Jb\,\backslash\,\Ab/\,\Ib_{\mu,\de^+}$ and is of highest weight with respect to the action of $\Y(\gl_n)$ on the same quotient space. \\
{\rm\,\,(iii)}
For any\/ $\si\in\Sg_m$ the operator \eqref{bbjioper} determined by\/ $\xic_\si$ maps the vector\/
$v_{\mu}^\la$ to the image in\/ $\Jb\,\backslash\,\Ab/\,\Ib_{\si\circ\mu,\si(\de^+)}$ of\/
$\si(v_{\mu}^\la)\in\Ab$ multiplied by $\prods{\eta\in\De_\si} \!z_\eta$.
\end{prop}

\begin{proof}
Part (i) follows directly from the definition of the ideal $\Ib_{\mu,\de^+}$. Let us now prove Part (ii). Subalgebra $\h$ acts on the quotient space $\Jb\,\backslash\,\Ab/\,\Ib_{\mu,\de^+}$ via left multiplication on $\Ab$. Let symbol $\equiv$ denote the equalities in $\Ab$ modulo the left ideal $\Ib_{\mu,\de^+}$. Then we have
\begin{gather*}
E_{aa} v_{\mu}^\la \;=\;
v_{\mu}^\la E_{aa} + \hs{\zeta_n(E_{aa})\,,\,v_{\mu}^\la} \,=\;
v_{\mu}^\la(E_{aa}\mp\nu_a) \;\equiv\;
v_{\mu}^\la\hr{\zeta_n(E_{aa})+\mu_a\mp\nu_a}
\;\equiv \\[4pt] \equiv\;
v_{\mu}^\la\hr{\mp\frac{n}2+\mu_a\mp\nu_a} \;=\;
\la_a v_{\mu}^\la
\end{gather*}
where one should choose the upper sign for $a=1,\dots,p$\/ and the lower sign for $a=p+1,\dots,m$. Next, $\Y(\gl_n)$-modules $\Jb\,\backslash\,\Ab/\,\Ib_{\mu,\de^+}$ and $\Ec_{p,q}(M_\mu)_\n$ are isomorphic. Recall that $T_{ij}(u)$ acts on $\Ec_{p,q}(M_\mu)_\n$ with the help of the comultiplication $\De$. Now Corollary~\ref{bimequiv} and formulas~\eqref{Tij} imply that $T_{ij}(u)v_{\mu}^\la = 0$ for all $1 \le i<j \le n$. Moreover, one can check that
$$
T_{ii}(u)v_\mu^\la =
v_\mu^\la \cdot
\left\{
\begin{aligned}
&\prodl{a=p+1}{m}\frac{u+\mu_a+\rho_a+\nu_a}{u+\mu_a+\rho_a}, & &i=1 \\
&\prodl{a=1}{p}\frac{u+\mu_a+\rho_a-\nu_a-1}{u+\mu_a+\rho_a}, & &i=n \\
&1, & &\text{otherwise.}
\end{aligned}
\right.
$$
Therefore, $v_\mu^\la$ is a highest weight vector with respect to the action of $\Y(\gl_n)$ on the quotient space $\Jb\,\backslash\,\Ab/\,\Ib_{\mu,\de^+}$.

We will prove Part (iii) by induction on the length of a reduced decomposition of $\si$. If $\si$ is the identity element of $\Sg_m$, then the required statement is tautological. Now suppose that for some $\si\in\Sg_m$ the statement of (iii) is true. Take any simple reflection $\si_a\in\Sg_m$ with $1\le a\le m-1$ such that $\si_a\si$ has a longer reduced decomposition in terms of $\si_1,\dots,\s_m$ than $\si$.

Consider the simple root $\eta_a$, corresponding to the reflection $\si_a$. Let $\eta=\si^{-1}(\eta_a)$ then
$$
\si_a\si(\eta)=\si_a(\eta_a)=-\eta_a\notin\De^+.
$$
Thus, $\De_{\si_a\si}=\De_\si\sqcup\hc{\eta}$. Let $\ka\in\h^\ast$ be the weight with the labels determined by \eqref{kappa}. Using the proof of Theorem~\ref{proposition4.5}, we get the equality of two left ideals of the algebra $\A$,
$$
\Ib_{\ts(\si_a\si)\circ\mu\ts,\ts(\si_a\si)(\de^+)} \,=\, \tilde\I_{\ts(\si_a\si)\circ\ka\ts,\ts(\si_a\si)(\de^+)}.
$$
Modulo the second of these two ideals the element $H_a$ equals
\be
\notag
((\si_a\si)\circ\ka)(H_a) \,=\, (\si_a\si(\ka+\rho)-\rho)(H_a) \,=\, (\ka+\rho)(\si^{-1}\si_a(H_a))-\rho(H_a) \,=\, \\[8pt]
-(\ka+\rho)(\si^{-1}(H_a))-1 \,=\, -(\ka+\rho)(H_\eta)-1 \,=\, -\mu^\ast_b+\mu^\ast_c-1.
\ee
Here $H_\eta = \si^{-1}(H_a)$ is the coroot corresponding to the root $\eta$, and we use the standard bilinear form on $\h^\ast$.

Let us use the statement of (iii) as the induction assumption. Denote $\de=\si(\de^+)$. Consider three cases.

I. Let $b,c = 1,\dots,p$, then $\de_a = \de'_a$ and $\de_{a+1} = \de'_{a+1}$. Hence,
$$
\si\hr{v_{\mu}^\la} = p_{an}^{\nu_b}p_{a+1\,n}^{\nu_c}Y
$$
where $Y$ is an element of the subalgebra of $\PD$ generated by all $x_{dk}$ and $\d_{dk}$ with $d\ne a,\,a+1$. Now we apply the Lemma~\ref{ppnorma} with $s=\nu_b$ and $t=\nu_c$. After substitution $-\mu^\ast_b + \mu^\ast_c - 1$ for $H_a$, the fraction in the lemma turns into
$$
\prodl{r=1}{\nu_b}\frac{-\mu^\ast_b+\mu^\ast_c +r}{-\mu^\ast_b+\mu^\ast_c+\nu_b-\nu_c-r} \;=\; \prodl{r=1}{\nu_b}\frac{\mu^\ast_b-\mu^\ast_c-r}{\la^\ast_b-\la^\ast_c+r}.
$$

II. Let $b,c = p+1,\dots,m$, then $\de_a = -\de'_a$ and $\de_{a+1} = -\de'_{a+1}$. Hence,
$$
\si\hr{v_{\mu}^\la} = q_{a1}^{\nu_b}q_{a+1\,1}^{\nu_c}Y
$$
where $Y$ is an element of the subalgebra of $\PD$ generated by all $x_{dk}$ and $\d_{dk}$ with $d\ne a,\,a+1$. Now we apply the Lemma~\ref{qqnorma} with $s=\nu_b$ and $t=\nu_c$. After substitution $-\mu^\ast_b + \mu^\ast_c - 1$ for $H_a$, the fraction in the lemma turns into
$$
\prodl{r=1}{\nu_c}\frac{-\mu^\ast_b+\mu^\ast_c+r}{-\mu^\ast_b+\mu^\ast_c+\nu_c-\nu_b-r} \;=\; \prodl{r=1}{\nu_c}\frac{\mu^\ast_b-\mu^\ast_c-r}{\la^\ast_b-\la^\ast_c+r}.
$$

III. Let $b=1,\dots,p, \; c=p+1,\dots,m$, then $\de_a=\de'_a$ and $\de_{a+1}=-\de'_{a+1}$. Hence,
$$
\si\hr{v_{\mu}^\la} = p_{an}^{\nu_b}q_{a+1\,1}^{\nu_c}Y
$$
where $Y$ is an element of the subalgebra of $\PD$ generated by all $x_{dk}$ and $\d_{dk}$ with $d\ne a,\,a+1$. Now we apply the Lemma~\ref{pqnorma} with $s=\nu_b$ and $t=\nu_c$. When $n=1$, after substitution $-\mu^\ast_b + \mu^\ast_c - 1$ for $H_a$, the fraction in the lemma turns into
$$
\prodl{r=1}{\upsilon}\frac{-\mu^\ast_b+\mu^\ast_c+r-1}{-\mu^\ast_b+\mu^\ast_c+\nu_b+\nu_c-r} \;=\; \prodl{r=1}{\upsilon}\frac{\mu^\ast_b-\mu^\ast_c-r+1}{\la^\ast_b-\la^\ast_c+r-1}.
$$

Thus, in the three cases considered above the operator
$$
\xic_{\si_a\si}\colon\Jb\,\backslash\,\Ab\,/\,\Ib_{\ts\mu,\de^+} \to\, \Ib_{(\si_a\si)\circ\mu\ts,\ts(\si_a\si)(\de^+)}
$$
maps the vector $v_{\mu}^\la$ to the image of
$$
\si_a\si(v_{\mu}^\la) \,\,\cdot\!\!\prods{\eta\in\De_{\si_a\si}} \!\!z_\eta \in\Ab
$$
in the vector space $\Jb\,\backslash\,\Ab\,/\,\Ib_{(\si_a\si)\circ\mu\ts,\ts(\si_a\si)(\de^+)}$. This observation makes the induction step.
\end{proof}

\subsection*{\it\normalsize 4.2.\ Skew-symmetric case}

In this subsection we consider only the case of anticommuting variables, hence from now on and till the end of this subsection we assume $\th=-1$. Proposition \ref{isisanticomm} below determines the image of the highest weight vector $v_\mu^\la$ of the $\Y(\gl_n)$-module $\Ec_{p,q}(M_\mu)_\n^\la$ under the action of the operator $\xic_\si$. The proof of the Proposition~\ref{isisanticomm} is based on the following three lemmas. Proofs of the first two of them are similar to the proof of Lemma 5.6 in \cite{KN4}. Proof of the last one is similar to the proof of Lemma 5.7 in \cite{KN4}. Let $s,t=1,\dots,n$, define
\begin{align*}
f_{as} = p_{a\,n-s+1} \dots p_{an} \qquad\text{and}\qquad g_{as} = q_{a1} \dots q_{as}.
\end{align*}

\begin{lem}
\label{antippnorma}
For any\/ $a=1,\dots,m-1$ the operator\/ $\xic_a$ on\/ $\Jb\,\backslash\,\Ab$ maps the image in\/ $\Jb\,\backslash\,\Ab$ of\/ $f_{as} f_{a+1\,t}\in\Ab$ to the image in\/ $\Jb\,\backslash\,\Ab$ of the product
$$
f_{a+1\,s} f_{at} \cdot
\left\{
\begin{aligned}
\frac{H_a+s-t+1}{H_a+1} & \qquad\text{if}\quad s>t, \\
1\hspace{30pt} &\qquad\text{if}\quad s \le t
\end{aligned}
\right.
$$
plus the images in\/ $\Jb\,\backslash\,\Ab$ of elements of the left ideal in\/ $\Ab$ generated by\/ $\Jpb$ and \eqref{qqideal}.
\end{lem}

\begin{lem}
\label{antiqqnorma}
For any\/ $a=1,\dots,m-1$ the operator\/ $\xic_a$ on\/ $\Jb\,\backslash\,\Ab$ maps the image in\/ $\Jb\,\backslash\,\Ab$ of\/ $g_{as} g_{a+1\,t}\in\Ab$ to the image in\/ $\Jb\,\backslash\,\Ab$ of the product
$$
g_{a+1\,s} g_{at} \cdot
\left\{
\begin{aligned}
\frac{H_a+t-s+1}{H_a+1} & \qquad\text{if}\quad s<t, \\
1\hspace{30pt} &\qquad\text{if}\quad s \ge t.
\end{aligned}
\right.
$$
plus the images in\/ $\Jb\,\backslash\,\Ab$ of elements of the left ideal in\/ $\Ab$ generated by\/ $\Jpb$ and \eqref{ppideal}.
\end{lem}

\begin{lem}
\label{antipqnorma}
For any\/ $a=1,\dots,m-1$ the operator\/ $\xic_a$ on\/ $\Jb\,\backslash\,\Ab$ maps the image in\/ $\Jb\,\backslash\,\Ab$ of\/ $f_{as} g_{a+1\,t}\in\Ab$ to the image in\/ $\Jb\,\backslash\,\Ab$ of the product
$$
f_{a+1\,s} g_{at} \cdot
\left\{
\begin{aligned}
\frac{H_a+s+t-n+1}{H_a+1} & \qquad\text{if}\quad s+t>n, \\
1\hspace{30pt} &\qquad\text{if}\quad s+t \le n.
\end{aligned}
\right.
$$
plus the images in\/ $\Jb\,\backslash\,\Ab$ of elements of the left ideal in\/ $\Ab$ generated by\/ $\Jpb$ and \eqref{pqideal}.
\end{lem}

We keep assuming that the weight $\mu$ satisfies the condition \eqref{notinso}. We also assume that $\nu_1,\dots,\nu_m \in \hc{0,1,\dots,n}$. Set
$$
\nu'_a =
\begin{cases}
n - \nu_a &\text{if} \quad a = 1 \sco p, \\
\nu_a &\text{if} \quad a = p+1 \sco m.
\end{cases}
$$
Let $(\mu^\ast_1,\dots,\mu^\ast_m)$ be the sequence of labels of weight $\mu+\rho+\frac{n}2\de'$. For each positive root $\eta=E_{bb}^\ast-E_{cc}^\ast\in\Delta^+$ with $1\le b<c \le m$ define a number $z_\eta\in\Cbb$ by
$$
z_\eta =
\begin{cases} \displaystyle
\frac{\la^\ast_b-\la^\ast_c}{\mu^\ast_b-\mu^\ast_c} &\text{if} \quad \nu'_b < \nu'_c, \\
\;\quad\; 1 &\text{otherwise}.
\end{cases}
$$
Let $v_{\mu}^\la$ denote the image of the vector
$$
\prodl{a=1}{p}\,f_{a\nu_a} \;\cdot\! \prodl{a=p+1}{m}\!\!g_{a\nu_a} \,\in\, \Ab
$$
in the quotient space $\Jb\,\backslash\,\Ab/\,\Ib_{\mu,\de^+}$.

\begin{prop} \label{isisanticomm}
{\rm\,\,(i)}
The vector $v_{\mu}^\la$ does not belong to the zero coset in $\Jb\,\backslash\,\Ab/\,\Ib_{\mu,\de^+}$. \\
{\rm\,\,(ii)}
The vector $v_{\mu}^\la$ is of weight $\lambda$ under the action of\/ $\h$ on\/
$\Jb\,\backslash\,\Ab/\,\Ib_{\mu,\de^+}$ and is of highest weight with respect to the action of $\Y(\gl_n)$ on the same quotient space. \\
{\rm\,\,(iii)}
For any\/ $\si\in\Sg_m$ the operator \eqref{bbjioper} determined by\/ $\xic_\si$ maps the vector\/
$v_{\mu}^\la$ to the image in\/ $\Jb\,\backslash\,\Ab/\,\Ib_{\si\circ\mu,\si(\de^+)}$ of\/
$\si(v_{\mu}^\la)\in\Ab$ multiplied by $\prods{\eta\in\De_\si} \!z_\eta$.
\end{prop}

\begin{proof}
Part (i) follows directly from the definition of the ideal $\Ib_{\mu,\de^+}$. Let us now prove Part (ii). Subalgebra $\h$ acts on the quotient space $\Jb\,\backslash\,\Ab/\,\Ib_{\mu,\de^+}$ via left multiplication on $\Ab$. Let symbol $\equiv$ denote the equalities in $\Ab$ modulo the left ideal $\Ib_{\mu,\de^+}$. Then we have
\begin{gather*}
E_{aa} v_{\mu}^\la \;=\;
v_{\mu}^\la E_{aa} + \hs{\zeta_n(E_{aa})\,,\,v_{\mu}^\la} \,=\;
v_{\mu}^\la(E_{aa}\mp\nu_a) \;\equiv\;
v_{\mu}^\la\hr{\zeta_n(E_{aa})+\mu_a\mp\nu_a}
\;\equiv \\[4pt] \equiv\;
v_{\mu}^\la\hr{\pm\frac{n}2+\mu_a\mp\nu_a} \;=\;
\la_a v_{\mu}^\la
\end{gather*}
where one should choose the upper sign for $a=1,\dots,p$\/ and the lower sign for $a=p+1,\dots,m$. Next, $\Y(\gl_n)$-modules $\Jb\,\backslash\,\Ab/\,\Ib_{\mu,\de^+}$ and $\Ec_{p,q}(M_\mu)_\n$ are isomorphic. Recall that $T_{ij}(u)$ acts on $\Ec_{p,q}(M_\mu)_\n$ with
the help of the comultiplication $\De$. Now, Corollary~\ref{bimequiv} and formulas~\eqref{Tij} imply that $T_{ij}(u)v_{\mu}^\la = 0$ for all $1 \le i<j \le n$. Moreover, one can check that
$$
T_{ii}(u)v_\mu^\la = v_\mu^\la \cdot
\prods{a \colon \nu'_a\ge i}\frac{u-\mu_a-\rho_a+1}{u-\mu_a-\rho_a}.
$$
Therefore, $v_{\mu}^\la$ is a highest weight vector with respect to the action of $\Y(\gl_n)$ on the quotient space $\Jb\,\backslash\,\Ab/\,\Ib_{\mu,\de^+}$.

We will prove Part (iii) by induction on the length of a reduced decomposition of $\si$. If $\si$ is the identity element of $\Sg_m$, then the required statement is tautological. Now suppose that for some $\si\in\Sg_m$ the statement of (iii) is true. Take any simple reflection $\si_a\in\Sg_m$ with $1\le a\le m-1$ such that $\si_a\si$ has a longer reduced decomposition in terms of $\si_1,\dots,\s_m$ than $\si$.

Consider the simple root $\eta_a$, corresponding to the reflection $\si_a$. Let $\eta=\si^{-1}(\eta_a)$ then
$$
\si_a\si(\eta)=\si_a(\eta_a)=-\eta_a\notin\De^+.
$$
Thus $\De_{\si_a\si}=\De_\si\sqcup\hc{\eta}$. Let $\ka\in\h^\ast$ be the weight with the labels determined by \eqref{kappa} with $\th=1$. Using the proof of Theorem~\ref{proposition4.5}, we get the equality of two left ideals of the algebra $\A$,
$$
\Ib_{\ts(\si_a\si)\circ\mu\ts,\ts(\si_a\si)(\de^+)} \,=\, \tilde\I_{\ts(\si_a\si)\circ\ka\ts,\ts(\si_a\si)(\de^+)}.
$$
Modulo the second of these two ideals the element $H_a$ equals
\be
\notag
((\si_a\si)\circ\ka)(H_a) \,=\, (\si_a\si(\ka+\rho)-\rho)(H_a) \,=\, (\ka+\rho)(\si^{-1}\si_a(H_a))-\rho(H_a) \,=\,  \\[8pt]
-(\ka+\rho)(\si^{-1}(H_a))-1 \,=\, -(\ka+\rho)(H_\eta)-1 \,=\, -\mu^\ast_b+\mu^\ast_c-1.
\ee
Here $H_\eta = \si^{-1}(H_a)$ is the coroot corresponding to the root $\eta$, and we use the standard bilinear form on $\h^\ast$.

Let us use the statement of (iii) as the induction assumption. Denote $\de=\si(\de^+)$. Consider three cases.

I. Let $b,c = 1,\dots,p$, then $\de_a = \de'_a$ and $\de_{a+1} = \de'_{a+1}$. Hence,
$$
\si\hr{v_{\mu}^\la} = f_{a\nu_b}f_{a+1\,\nu_c}Y
$$
where $Y$ is an element of the subalgebra of $\GD$ generated by all $x_{dk}$ and $\d_{dk}$ with $d\ne a,\,a+1$. Now we apply the Lemma~\ref{antippnorma} with $s=\nu_b$ and $t=\nu_c$. After substitution $-\mu^\ast_b + \mu^\ast_c - 1$ for $H_a$, the fraction in the lemma turns into
$$
\frac{-\mu^\ast_b+\mu^\ast_c+\nu_b-\nu_c}{-\mu^\ast_b+\mu^\ast_c} \;=\; \frac{\la^\ast_b-\la^\ast_c}{\mu^\ast_b-\mu^\ast_c}.
$$

II. Let $b,c = p+1,\dots,m$, then $\de_a = -\de'_a$ and $\de_{a+1} = -\de'_{a+1}$. Hence,
$$
\si\hr{v_{\mu}^\la} = g_{a\nu_b}g_{a+1\,\nu_c}Y
$$
where $Y$ is an element of the subalgebra of $\GD$ generated by all $x_{dk}$ and $\d_{dk}$ with $d\ne a,\,a+1$. Now we apply the Lemma~\ref{antiqqnorma} with $s=\nu_b$ and $t=\nu_c$. After substitution $-\mu^\ast_b + \mu^\ast_c - 1$ for $H_a$, the fraction in the lemma turns into
$$
\frac{-\mu^\ast_b+\mu^\ast_c+\nu_c-\nu_b}{-\mu^\ast_b+\mu^\ast_c} \;=\; \frac{\la^\ast_b-\la^\ast_c}{\mu^\ast_b-\mu^\ast_c}.
$$

III. Let $b=1,\dots,p, \; c=p+1,\dots,m$, then $\de_a=\de'_a$ and $\de_{a+1}=-\de'_{a+1}$. Hence,
$$
\si\hr{v_{\mu}^\la} = f_{a\nu_b}g_{a+1\,\nu_c}Y
$$
where $Y$ is an element of the subalgebra of $\GD$ generated by all $x_{dk}$ and $\d_{dk}$ with $d\ne a,\,a+1$. Now we apply the Lemma~\ref{antipqnorma} with $s=\nu_b$ and $t=\nu_c$. When $n=1$, after substitution $-\mu^\ast_b + \mu^\ast_c - 1$ for $H_a$, the fraction in the lemma turns into
$$
\frac{-\mu^\ast_b+\mu^\ast_c+\nu_b+\nu_c-n}{-\mu^\ast_b+\mu^\ast_c} \;=\; \frac{\la^\ast_b-\la^\ast_c}{\mu^\ast_b-\mu^\ast_c}.
$$

Thus, in the three cases considered above the operator
$$
\xic_{\si_a\si}\colon\Jb\,\backslash\,\Ab\,/\,\Ib_{\ts\mu,\de^+} \to\, \Ib_{(\si_a\si)\circ\mu\ts,\ts(\si_a\si)(\de^+)}
$$
maps the vector $v_{\mu}^\la$ to the image of
$$
\si_a\si(v_{\mu}^\la) \,\,\cdot\!\!\prods{\eta\in\De_{\si_a\si}} \!\!z_\eta \in\Ab
$$
in the vector space $\Jb\,\backslash\,\Ab\,/\,\Ib_{(\si_a\si)\circ\mu\ts,\ts(\si_a\si)(\de^+)}$. This observation makes the induction step.
\end{proof}

\section*{\bf\normalsize 5.\ Conjecture}
\setcounter{section}{5}
\setcounter{theorem}{0}
\setcounter{equation}{0}

The formula \eqref{q1} yields that for any vector $v\in\Jb\,\backslash\,\Ab\,/\,\Ib_{\ts\mu,\de^+}$ the image of $v$ under the operator $\xic_\si$ is well defined unless $\,(H_a+1-r)v = 0\,$ for some integer $r$. As we have shown in the previous section $H_a$ acts as $\,-\mu^\ast_b+\mu^\ast_c-1\,$ for some $\,1\le b<c\le m\,$ on the target module in the mapping~\eqref{distoperla}. It follows from relations \eqref{Ore} that after commuting $H_a$ with $v$, we will get $v(\la^\ast_b-\la^\ast_c+r)$. Therefore, the operators $\xic_\si$ are well defined unless
\beq \label{weight_cond}
\la_b^\ast - \la_c^\ast \;=\; -1,-2,\dots \qquad\text{for some}\qquad 1\le b<c\le m.
\eeq
In the latter case it can be shown that a factor of module \eqref{source} by the kernel of operator $\xic_{\si_0}$ is an irreducible (and non-zero under certain condition on the numbers $\nu_1,\dots,\nu_m$) $\Y(\gl_n)$-module, where $\si_0$ is the longest element of the Weyl group (see \cite{KN5}).

Recall the operator $\xi'_\si$ defined right before Theorem~\ref{th_inter}. Now, we state
\begin{conj} \label{conj}
Every irreducible finite-dimensional rational module of the Yangian $\Y(\gl_n)$ may be obtained as the factor of the module \eqref{rational_source} over the kernel of the intertwining operator $\xi'_{\si_0}$ for some weights $\mu$ and $\la$, with $\la$ satisfying condition \eqref{weight_cond}.
\end{conj}

\section*{\bf\normalsize Appendix}

\renewcommand{\theequation}{A.{\arabic{equation}}}
\renewcommand{\thetheorem}{A.{\arabic{theorem}}}
\setcounter{theorem}{0}
\setcounter{equation}{0}

Let us first prove some properties of the matrix $X(u)$ defined by \eqref{Xu}.

\begin{prop} \label{A1}
\hskip5pt$\rm i)$ The following relation holds:
\beq \label{XX}
(u-v)\cdot X(u) X(v) = X(v) - X(u)\,;
\eeq
$\rm ii)$ The elements $X_{ab}(u)$ satisfy the Yangian relation:
\beq \label{X_yang}
(u-v)\cdot[X_{ab}(u),X_{cd}(v)]=\th(X_{cb}(u)X_{ad}(v)-X_{cb}(v)X_{ad}(u)).
\eeq
\end{prop}

\begin{proof}
The part i) follows from equality
$$
(u-v) = \hr{u + \th\Ep} - \hr{v + \th\Ep}
$$
multiplied by $X(u)$ from the left and by $X(v)$ from the right. Let us start the proof of the part~ii) with equality
$$
\hs{\hr{u + \th\Ep}_{ef}, \hr{v + \th\Ep}_{gh}} = \de_{eh} E_{fg} - \de_{fg}E_{he}.
$$
We multiply the above equality by $X_{ae}(u)$ from the left, by $X_{fb}(u)$ from the right, and take a sum over indices $e,f:$
\begin{align*}
\suml{e,f=1}{m} &\hr{X_{ae}(u) \hr{u + \th\Ep}_{ef} \hr{v + \th\Ep}_{gh} X_{fb}(u) -
X_{ae}(u) \hr{v + \th\Ep}_{gh} \hr{u + \th\Ep}_{ef} X_{fb}(u)} = \\
= &\suml{e,f=1}{m} \hr{X_{ae}(u) \de_{eh} E_{fg} X_{fb}(u) - X_{ae}(u) \de_{fg} E_{he} X_{fb}(u)}.
\end{align*}
Thus, we get
$$
\hs{\hr{v + \th\Ep}_{gh}, X_{ab}(u)} = \suml{e,f=1}{m} \hr{X_{ah}(u) E_{fg} X_{fb}(u) - X_{ae}(u) E_{he} X_{gb}(u)}.
$$
Using that
$$
\suml{f=1}{m} E_{fg} X_{fb}(u) = \th\hr{\de_{bg} - u X_{gb}(u)}  \qquad\text{and}\qquad  \suml{e=1}{m} X_{ae}(u) E_{he} = \th\hr{\de_{ah} - u X_{ah}(u)},
$$
we obtain
$$
\hs{\hr{v + \th\Ep}_{gh}, X_{ab}(u)} = \th\hr{\de_{bg} X_{ah}(u) - \de_{ah} X_{gb}(u)}.
$$
Multiplying the above equation by $X_{cg}(v)$ from the left, by $X_{hd}(v)$ from the right, and taking a sum over indices $g,h$, we arrive to equality
$$
[X_{ab}(u), X_{cd}(v)] = \th\suml{g,h=1}{m}\hs{X_{cb}(v) X_{ah}(u) X_{hd}(v) - X_{cg}(v) X_{gb}(u) X_{ad}(v)}.
$$
Now using the result of the part i), we get
$$
(u-v) \cdot [X_{ab}(u), X_{cd}(v)] \;=\;
\th\Br{X_{cb}(v)\br{X_{ad}(v)-X_{ad}(u)}-\br{X_{cb}(v)-X_{cb}(u)}X_{ad}(v)}
$$
and the statement of the part ii) follows.
\end{proof}

{\flushleft\it Proof of Proposition~\ref{prop_bimod}:}

\vspace{5pt}
We prove the part i) by direct calculation. During the proof we will write $T_{ij}(u)$ instead of its image under $\al_m$ in the algebra $\U(\gl_m)\otimes\HD$. Using relations \eqref{E_rel2}, \eqref{E_rel3}, we get
\begin{align*}
(u-v)\cdot\hs{T_{ij}(u),T_{kl}(v)} = \\ =
\suml{a,b,c,d=1}{m} (u-v)\cdot
\Br{
&X_{ab}(u)X_{cd}(v) \otimes
(\th\Eh_{ck,bj}\Eh_{ai,dl} + \de_{bc}\de_{jk}\Eh_{ai,dl} - \th\de_{ab}\de_{ij}\Eh_{ck,dl}) - \\[-4pt] -
&X_{cd}(v)X_{ab}(u) \otimes
(\th\Eh_{ck,bj}\Eh_{ai,dl} + \de_{ad}\de_{il}\Eh_{ck,bj} - \th\de_{ab}\de_{ij}\Eh_{ck,dl})
} = \\ =
\suml{a,b,c,d=1}{m} (u-v)\cdot
\Br{
&[X_{ab}(u),X_{cd}(v)] \otimes \th(\Eh_{ck,bj}\Eh_{ai,dl} - \de_{ab}\de_{ij}\Eh_{ck,dl}) + \\[-4pt] +
&X_{ab}(u)X_{cd}(v) \otimes \de_{bc}\de_{jk}\Eh_{ai,dl} - X_{cd}(v)X_{ab}(u) \otimes \de_{ad}\de_{il}\Eh_{ck,bj}
}.
\end{align*}
Next, using relations \eqref{XX} and \eqref{X_yang}, we obtain
\begin{gather*}
(u-v)\cdot\hs{T_{ij}(u),T_{kl}(v)} = \\[4pt]
= \suml{a,b,c,d=1}{m} \bbr{\Br{X_{cb}(u)X_{ad}(v) - X_{cb}(v)X_{ad}(u)}
\otimes (\Eh_{ck,bj}\Eh_{ai,dl} - \de_{ab}\de_{ij}\Eh_{ck,dl}) + \\[-4pt]
+ (u-v) \cdot \Br{ X_{ab}(u)X_{cd}(v) \otimes \de_{bc}\de_{jk}\Eh_{ai,dl}
- X_{cd}(v)X_{ab}(u) \otimes \de_{ad}\de_{il}\Eh_{ck,bj}}}.
\end{gather*}
Now, using the definition of the homomorphism $\al_m$, we get
\begin{gather*}
\br{T_{kj}(u) - \de_{jk}} \br{T_{il}(v) - \de_{il}} -
\br{T_{kj}(v) - \de_{jk}} \br{T_{il}(u) - \de_{il}} - \\[6pt]
- \de_{ij} \suml{c,d=1}{m} \Br{\br{X(u)X(v)}_{cd} -
\br{X(v)X(u)}_{cd}} \otimes \Eh_{ck,dl} + \\[-4pt]
+ \de_{jk} \suml{a,d=1}{m} (u-v)\br{(X(u)X(v))_{ad}} \otimes \Eh_{ai,dl} +
\de_{il} \suml{b,c=1}{m} (u-v)\br{(X(v)X(u))_{cb}} \otimes \Eh_{ck,bj}.
\end{gather*}
It follows from relation \eqref{XX} that $X(u)X(v) = X(v)X(u).$ Therefore, we finish the proof of the part i) by the following calculations:
\begin{gather*}
(u-v)\cdot\hs{T_{ij}(u),T_{kl}(v)} \quad=\quad T_{kj}(u)T_{il}(v) - T_{kj}(v)T_{il}(u) \;-\\[12pt]
-\;\de_{jk} \br{T_{il}(v) - T_{il}(u)} \,-\, \de_{il}\br{T_{kj}(u) - T_{kj}(v)} \;+ \\[12pt]
+\;\de_{jk} \br{T_{il}(v) - T_{il}(u)} \,+\, \de_{il}\br{T_{kj}(u) - T_{kj}(v)} \;= \\[12pt]
= \; T_{kj}(u)T_{il}(v) - T_{kj}(v)T_{il}(u).
\end{gather*}
The part ii) is also proved by straightforward verification:
\begin{gather*}
\BS{ E_{cd} \otimes 1 + 1 \otimes \zeta_n(E_{cd}), T_{ij}(u) } = \\
= \;\BS{ E_{cd} \otimes 1, \suml{a,b=1}{m} X_{ab}(u) \otimes \Eh_{ai,bj} } \,+\,
\BS{ 1 \otimes \suml{k=1}{n} \Eh_{ck,dk}, \suml{a,b=1}{m} X_{ab}(u) \otimes \Eh_{ai,bj} } \;=
\end{gather*}
\begin{gather*}
= \suml{a,b=1}{m} \Br{ \de_{bd}X_{ac}(u) - \de_{ac}X_{db}(u) } \otimes \Eh_{ai,bj} \,+
\suml{k=1}{n} \suml{a,b=1}{m} \Br{X_{ab}(u) \otimes \hr{\de_{ad}\de_{ik}\Eh_{ck,bj} - \de_{bc}\de_{jk}\Eh_{ai,dk}}} = \\
= \;\suml{a=1}{m} X_{ac}(u) \otimes \Eh_{ai,dj} \,-\, \suml{b=1}{m} X_{db}(u) \otimes \Eh_{ci,bj} \,+ \suml{b=1}{m} X_{db}(u) \otimes \Eh_{ci,bj} \,-\, \suml{a=1}{m} X_{ac}(u) \otimes \Eh_{ai,dj} \;=\, 0.
\end{gather*}
$\hfill\square$

\section*{\bf\normalsize Acknowledgements}

I want to express my gratitude to Sergey Khoroshkin for introducing me to this subject, constant interest to my progress, and his many invaluable remarks and corrections. I am also very grateful to Maxim Nazarov for many fruitful discussions. Finally, I would like to thank the reviewer for many useful comments. The author was supported by RFBR grant 11-01-00962, joint CNRS-RFBR grant 09-01-93106-NCNIL, Federal Agency for Science and Innovations of Russian Federation under contract 02.740.11.5194, and by Russian Federal Nuclear Energy Agency under contract H.4e.45.90.11.1059.

\end{document}